\documentclass{article}

\usepackage[colorlinks=true,linkcolor=blue,citecolor=blue]{hyperref}
\newcommand{\email}[1]{\href{mailto:#1}{\protect\nolinkurl{#1}}}

\usepackage{graphicx}
\usepackage{color}
\usepackage{amsmath,amsfonts,amssymb,amsthm,bm}
\usepackage[a4paper]{geometry}

\usepackage{algorithm,algorithmic}

\newcommand{\diff}[1]{\,\mathrm{d}#1}
\let\origvec\vec
\newcommand{\vecc}[1]{\origvec{#1}}
\renewcommand{\vec}[1]{\boldsymbol{#1}}

\newcommand{\N}{\mathbb{N}}
\newcommand{\Z}{\mathbb{Z}}
\newcommand{\C}{\mathbb{C}}
\newcommand{\R}{\mathbb{R}}
\newcommand{\F}{\mathbb{F}}

\newcommand{\imagunit}{\mathrm{i}}
\newcommand{\twopii}{2 \pi \imagunit \;}

\DeclareMathOperator{\wal}{wal}

\newcommand{\floor}[1]{\left\lfloor #1 \right\rfloor}

\newcommand{\nzd}[1]{\##1} % number of non-zero digits
\newcommand{\Wspace}{\mathcal{W}_{\alpha,s,\vec{\gamma}}}
\newcommand{\ralpha}{r_\alpha(\vec{\gamma},\vec{k})}
\newcommand{\Dual}{\mathcal{D}}
\newcommand{\Qnet}{Q^{\mathrm{net}}}
\newcommand{\transp}{\top}
\newcommand{\x}{X}
\DeclareMathOperator*{\argmin}{argmin}

\newcommand{\mg}{\times} % multiplicative group

\def\citep#1#2{\cite[{#1}]{#2}}

\theoremstyle{plain}
  \newtheorem{theorem}{Theorem}
  \newtheorem{lemma}{Lemma}
  \newtheorem{corollary}{Corollary}
\theoremstyle{definition}
  \newtheorem{definition}{Definition}
  
\theoremstyle{remark}

\newcommand{\RefEq}[1]{~\textup{(\ref{#1})}}
\newcommand{\RefDef}[1]{Definition~\textup{\ref{#1}}}
\newcommand{\RefSec}[1]{Section~\textup{\ref{#1}}}

\newcommand{\RefThm}[1]{Theorem~\textup{\ref{#1}}}
\newcommand{\RefLem}[1]{Lemma~\textup{\ref{#1}}}
\newcommand{\RefCol}[1]{Corollary~\textup{\ref{#1}}}
\newcommand{\RefAlg}[1]{Algorithm~\textup{\ref{#1}}}

\allowdisplaybreaks

\begin{document}

%\title{Fast component-by-component construction of multivariate higher order polynomial lattice rules and efficient calculations of the worst-case error}
\title{Efficient calculation of the worst-case error \\ 
       and (fast) component-by-component construction \\
       of higher order polynomial lattice rules}
\author{
Jan Baldeaux\thanks{Jan Baldeaux, School of Finance and Economics, The University of Technology, Sydney, NSW 2007, Australia. \email{jan.baldeaux@uts.edu.au}} \and
Josef Dick\thanks{Josef Dick, School of Mathematics and Statistics, The University of New South Wales, Sydney, NSW 2052, Australia. \email{josef.dick@unsw.edu.au} \quad J.D. is supported by a Queen Elizabeth 2 Fellowship from the Australian Research Council.} \and
Gunther Leobacher\thanks{Gunther Leobacher, Institut f\"{u}r Finanzmathematik, Universit\"{a}t Linz, Altenbergerstra{\ss}e 69, A-4040 Linz, Austria. \email{gunther.leobacher@jku.at} \quad G.L. is partially supported by the Austrian Science Foundation (FWF), Project P21196.} \and
Dirk Nuyens\thanks{Dirk Nuyens, Department of Computer Science, K.U.Leuven, Celestijnenlaan 200A -- bus 2402, 3001 Heverlee, Belgium. \email{dirk.nuyens@cs.kuleuven.be} \quad D.N. is a postdoctoral fellow of the Research Foundation Flanders (FWO).} \and
Friedrich Pillichshammer\thanks{Friedrich Pillichshammer, Institut f\"{u}r Finanzmathematik, Universit\"{a}t Linz, Altenbergerstra{\ss}e 69, A-4040 Linz, Austria.
\email{friedrich.pillichshammer@jku.at} \quad F.P. is partially supported by the Austrian Science Foundation (FWF), Project S9609, that is part of the Austrian National Research Network ``Analytic Combinatorics and Probabilistic Number Theory''.}
}
%\institute{
%Dept. of Computer Science, K.U.Leuven, B-3001 Heverlee, Belgium
%\texttt{dirk.nuyens@cs.kuleuven.be}, \texttt{ronald.cools@cs.kuleuven.be}
%}
%\keywords      {Numerical integration, Quadrature, Cubature, Quasi-Monte Carlo methods,
%                Lattice rules, Digital nets, Fast component-by-component construction}
\maketitle

\begin{abstract}
We show how to obtain a fast component-by-component construction algorithm for higher order polynomial lattice rules. Such rules are useful for multivariate quadrature of high-dimensional smooth functions over the unit cube as they achieve the near optimal order of convergence. The main problem addressed in this paper is to find an efficient way of computing the worst-case error.
A general algorithm is presented and explicit expressions for base~2 are given.
To obtain an efficient component-by-component construction algorithm we exploit the structure of the underlying cyclic group.

We compare our new higher order multivariate quadrature rules to existing quadrature rules based on higher order digital nets by computing their worst-case error. These numerical results show that the higher order polynomial lattice rules improve upon the known constructions of quasi-Monte Carlo rules based on higher order digital nets. %Some higher order polynomial lattice rules in up to $10$~dimensions are presented for the benefit of the reader.
\end{abstract}

\vspace{.5cm}



\section{Introduction}

In this paper we are concerned with quasi-Monte Carlo rules, which are equal weight multivariate quadrature rules (or cubature rules)
\begin{align}\label{eq:qmc-rule}
  Q(f)
  &:=
  \frac{1}{N} \sum_{h=0}^{N-1} f(\vec{x}_h)
  ,
\end{align}
used to approximate multivariate integrals over the $s$-dimensional unit cube
\begin{align}\label{eq:I}
  I(f)
  &:=
  \int_{[0,1]^s} f(\vec{x}) \diff{\vec{x}}
  .
\end{align}
In contrast to the Monte Carlo method, which samples the function $f$ randomly in its domain, the integration nodes $\{\vec{x}_h\}_{h=0}^{N-1}$ used in the quasi-Monte Carlo rule $Q$ are chosen deterministically. The convergence of the error of the Monte Carlo method is $O(N^{-1/2})$ (in distribution), while the worst-case error for quasi-Monte Carlo is $O(N^{-\alpha} (\log N)^{s \alpha})$ \cite{Dic2007, Dic2008, DP2010, Nie92} and the random-case error for randomized quasi-Monte Carlo is $O(N^{-\alpha-1/2} (\log N)^{s (\alpha+1)})$ \cite{Dic2011, Owe97}. The latter two methods require that the integrand has smoothness $\alpha \ge 1$ (which means for instance that the integrand has square integrable partial mixed derivatives up to order $\alpha$ in each variable), whereas the Monte Carlo method requires only that the integrands have finite variance.

Different types of point sets for quasi-Monte Carlo rules exist. Those of interest here are \emph{digital nets} \cite{DP2010,Nie92}. These can be divided into \emph{\textup{(}classical\textup{)} digital nets} \cite{Nie92}, which achieve a convergence rate of $O(N^{-1} (\log N)^s)$ \cite{Nie92} for integrands of bounded variation, and their extension called \emph{higher order digital nets} \cite{Dic2007, Dic2008}, which achieve a convergence rate of $O(N^{-\alpha} (\log N)^{\alpha s})$ for integrands which have square integrable partial mixed derivatives of order $\alpha > 1$.
In \cite[Section~4.4]{Dic2008} an explicit method for constructing such higher order digital nets, based on a classical digital net, can be found. The method in the current paper gives an alternative construction for a higher order version of a specific type of digital net, namely \emph{polynomial lattice point sets}, see \cite[Section~4.4]{Nie92} or \cite[Chapter~10]{DP2010}. Constructions of classical polynomial lattice point sets based on a worst-case error criterion have previously been studied in \cite{DKPS2005}. Note that we call a quasi-Monte Carlo rule whose underlying quadrature points are polynomial lattice points a polynomial lattice rule.

Polynomial lattice point sets were generalized in \cite{DP2007} to obtain \emph{higher order polynomial lattice point sets}. In \cite{DKPS2007} existence results on higher order polynomial lattice point sets were compared to the explicit construction of higher order digital nets \cite{Dic2008} in terms of their $t$-value (a certain quality measure). For some values of dimension $s$ and/or smoothness parameter $\alpha$ the higher order polynomial lattice point sets have a better existence bound than the best results which can currently be obtained using the explicit construction of higher order digital nets from \cite{Dic2009}. The same is also true for classical digital nets, see \cite{LLNS96, Sch2000}. These findings motivated the quest for an explicit construction of higher order polynomial lattice rules in \cite{BDGP2011}. The construction employed there is an algorithm originally proposed for the construction of (integer) lattice rules, namely the \emph{component-by-component construction} algorithm, see, e.g., \cite{Kor59, Kor63, SR2002}. The higher order polynomial lattice rules so constructed achieve nearly optimal rates of convergence. For analogous results on polynomial lattice point sets see \cite{DKPS2005} (and also \cite{DP2005} for more background).

Straightforward implementation of the component-by-component (CBC) algorithm is however very costly with respect to computational time, hence methods for reducing the computational cost are needed. The fast component-by-component algorithm, introduced in \cite{NC2006-prime}, uses fast Fourier transforms (FFTs) to speed up the calculations. Some notes concerning the application of the fast algorithm to the construction of polynomial lattice rules were already made in \cite{NC2006-kernels}, with a more detailed analysis in \cite{Nuy2007}; see also \cite[Section~10.3]{DP2010}. In this paper we will adapt the fast algorithm for higher order polynomial lattice rules. To do so, we find a closed form for the worst-case error of our function space (where we consider the worst-case error as a function of the quadrature points). We show that our algorithm has a computational cost of $O(s N^\alpha \alpha \log N)$ using $O(N^\alpha)$ memory, compared to $O(s^2 N^{\alpha+1})$ for the straightforward implementation of the algorithm in \cite{BDGP2011}, using the same amount of memory. This speedup makes it possible to obtain higher order polynomial lattice rules for moderate dimensions and numbers of points. In the section on numerical results we provide constructions of higher order polynomial lattice rules in base $b = 2$ up to dimension $10$ and up to $4096$ points.
These numbers could be increased with more computational effort, but we have to remark that the search space grows exponentially with respect to the smoothness parameter~$\alpha$.

The efficient calculation of the worst-case error of our function space is an essential ingredient in such an algorithm. We show that the kernel function associated with the worst-case error can be evaluated at a point $x$ in time $O(\alpha n)$, where  $\alpha$ is the smoothness of the space and $x$ is a rational number $v/b^n$, $0 \le v < b^n$. Moreover, in the case of the greatest practical importance, i.e., where the base equals $2$, we show explicit expressions for smoothness~$2$ and~$3$ which are exact for any real $x \in [0,1)$ (see \RefCol{col:base2}).

We compare the performance of higher order polynomial lattice rules
constructed using our fast component-by-component algorithm to the
explicit construction as outlined in \cite{Dic2008} and find that the new algorithm performs better in the cases considered. Finally, for the benefit of the reader, we present some limited tables of higher order polynomial lattice rules constructed using the fast component-by-component algorithm, allowing the reader to apply the rules presented in this paper to problems of interest and to verify implementations of the algorithm.

In the next section we provide the reader with some background, and notation, on Walsh spaces, digital nets, and the worst-case error. More detailed information can be found in \cite{DP2010} and 
%in the prequel to this paper 
\cite{BDGP2011}, where also bounds on the worst-case error for higher order polynomial lattice rules were proven. In Section~\ref{sec:fastcbc} we show how to efficiently calculate the worst-case error and how the construction of higher order polynomial lattice rules can be done using the fast component-by-component approach of \cite{Nuy2007, NC2006-prime}.

\section{Background}\label{sec:funspace}

We first introduce some notation. Let $\N_0 = \{0, 1, 2, \ldots\}$ denote the set of non-negative integers and $\N = \{1, 2, 3, \ldots\}$ the set of positive integers.
Further we need to be able to consider a non-negative integer $k \in \N_0$ in its unique base~$b$ representation:
\begin{align}\label{eq:k-base-b}
  k
  &=
  (\kappa_a \ldots \kappa_0)_b
  =
  \sum_{i=0}^{a} \kappa_i \, b^i
  ,
\end{align}
where $\kappa_i \in \{0, \ldots, b-1\}$ are the base $b$ digits of $k$ and $\kappa_a \ne 0$; $a=0$ for $x=0$.
Note that the base, $b$, is considered a fixed integer throughout. Moreover, in the further development in this paper, $b$ will be prime.
We will be specifically interested in the non-zero base $b$ digits of $k$.
The number of non-zero base $b$ digits of an integer $k$ will be denoted by $\nzd{k}$; where $\nzd{0} = 0$.
We can then represent $k \in \N_0$ uniquely as
\begin{align}\label{eq:k-base-b-nz}
  k
  &=
  \sum_{i=1}^{\nzd{k}} \kappa_{a_i} \, b^{a_i}
  ,
\end{align}
where now $\kappa_{a_i} \in \{1, \ldots, b-1\}$ and we demand $a_1 > \cdots > a_{\nzd{k}} \ge 0$.
Thus $\kappa_{a_1}$ is the most significant base $b$ digit of $k$.
For real $x \in [0,1)$ we write its base~$b$ representation
\begin{align}\label{eq:x-base-b}
  x
  &=
  (0.\xi_1 \xi_2 \ldots)_b
  =
  \sum_{i=1}^\infty \xi_i \, b^{-i}
  ,
\end{align}
where $\xi_i \in \{0, \ldots, b-1\}$.
This representation is unique in the sense that we do not allow an infinite repetition of the digit $b-1$ to the right.

\subsection{A function space based on Walsh series}

For $k \in \N_0$ the one-dimensional $k$th Walsh function in base $b$, $\wal_k : [0, 1) \to \C$, is defined by
\begin{align}\label{eq:walsh-1d}
  \wal_k(x)
  &:=
  \exp(\twopii (\xi_1 \kappa_0 + \cdots + \xi_{a+1} \kappa_a) / b)
  ,
\end{align}
where we have used the base $b$ digits of $x$ and $k$ as given in\RefEq{eq:k-base-b} and\RefEq{eq:x-base-b}. Note that Walsh functions (in base $b$) are piecewise constant functions.
For dimensions $s \ge 2$ and vectors $\vec{k} = (k_1, \ldots, k_s) \in \N_0^s$ and $\vec{x} = (x_1, \ldots, x_s) \in [0,1)^s$ we define $\wal_{\vec{k}} : [0, 1)^s \to \C$ as
\begin{align*}
  \wal_{\vec{k}}(\vec{x})
  &:=
  \prod_{j=1}^s \wal_{k_j}(x_j)
  .
\end{align*}
%(It will be clear from the context wether, e.g., $k_j$ denotes the $j$th entry of the vector $\vec{k}$, or the $j$th digit of the base $b$ representation of the scalar $k$.)

The integrand functions in this paper are assumed to have an absolutely convergent Walsh series representation
\begin{align*}
  f(\vec{x})
  &=
  \sum_{\vec{k} \in \N_0^s} \widehat{f}_{\vec{k}} \, \wal_{\vec{k}}(\vec{x})
  ,
\end{align*}
where the Walsh coefficients $\widehat{f}_{\vec{k}}$ are given by
\begin{align*}
  \widehat{f}_{\vec{k}}
  &=
  \int_{[0,1]^s} f(\vec{x}) \, \overline{\wal_{\vec{k}}(\vec{x})} \diff{\vec{x}}
  .
\end{align*}
Note that the Walsh functions form a complete orthonormal system of $L_2([0,1]^s)$. For more information on Walsh functions and their properties we refer to \cite[Chapter~14 and Appendix~A]{DP2010}.

In the following we define a function space by demanding a certain decay rate of the Walsh coefficients. To do so, we introduce some further notation. We define, for a fixed integer $\alpha > 1$ and a fixed sequence of positive weights $\vec{\gamma} = \{\gamma_1, \gamma_2, \ldots\}$ (in the sense of \cite{SW98}),
\begin{align}\label{eq:ralpha}
  \ralpha
  &:=
  \prod_{j \in u} \gamma_j \, r_\alpha(k_j)
  ,
  &
  r_\alpha(k)
  &:=
  b^{-\sum_{i=1}^{\min(\nzd{k},\alpha)} (a_i + 1)}
  ,
\end{align}
where for $\boldsymbol{k} = (k_1,\ldots, k_s)$ we set $u = \{1 \le j \le s: k_j\neq 0\}$ and where we used the first $\alpha$ positions $a_1+1,\ldots,a_{\nzd{k}}+1$ of the non-zero base $b$ digits of $k$ with the notation defined in\RefEq{eq:k-base-b-nz} in the one-dimensional definition of $r_\alpha$. For $\boldsymbol{k} = \boldsymbol{0} = (0,\ldots, 0)$ we set
\begin{equation*}
r_{\alpha}(\boldsymbol{0}) = 1.
\end{equation*}

We are now ready to specify which functions are in our function space $\Wspace$, which was also used in \cite{BDGP2011, Dic2008}.
For functions $f \in \Wspace$ we define the norm
\begin{align}\label{eq:norm}
 \| f \|_{\Wspace}
 &:=
 \sup_{\vec{k} \in \N_0^s} \frac{|\widehat{f}_{\vec{k}}|}{\ralpha}.
\end{align}
Then $\Wspace$ consists of all functions $f \in L_2([0,1]^s)$ for which $\|f\|_{\Wspace} < \infty$. The Walsh coefficients of $f \in \Wspace$ therefore satisfy a certain decay criterion, namely
\begin{align}\label{eq:bounded-Walsh-coeff}
  |\widehat{f}_{\vec{k}}|
  &
  \le
  \| f \|_{\Wspace} \; \ralpha
  \quad\forall\vec{k} \in \N_0^s.
\end{align}
%Note that the function space $\Wspace$ we have just defined is not a reproducing kernel Hilbert space, but just a Banach space\footnote{Technically we should first divide out functions with the same semi-norm, or extend the semi-norm to a norm by also considering $\vec{k} = \vec{0}$ in definition\RefEq{eq:norm} before we can speak about a Banach space. The motivation for working with a semi-norm is given by the fact that we are considering equal weight rules which integrate constant functions exactly. Therefore it is more useful to define this semi-norm which will be used in the forthcoming error bound\RefEq{eq:error-bound-for-f}.}.

It is clear that larger values of $\alpha$ might increase the norm of a function $f$, i.e., $\|f\|_{\Wspace} \le \|f\|_{\mathcal{W}_{\alpha',s,\boldsymbol{\gamma}}}$ for $\alpha \le \alpha'$.
The weights $\gamma_1, \gamma_2, \ldots$ are used to describe how  anisotropic the space is.
Usually it is assumed that $\gamma_1 \ge \gamma_2 \ge \cdots \ge 0$, meaning that the first dimension is more important than the second one and so on. Under certain conditions on these weights, it can be shown that numerical integration is tractable in the number of dimensions, see, e.g., \cite{DP2010,NW2008,NW2010}.

%The study of the fast component-by-component algorithm, which we are going to engage in in this paper, was primarily done in a reproducing kernel Hilbert space setting, see, e.g., \cite{NC2006-prime}. However, this will not be inconvenient here since we will be able to write the expression for the worst-case error in a similar form as for reproducing kernel Hilbert spaces. As such, we will be able to apply the fast component-by-component algorithm. The details will be given in \RefSec{sec:cbc}.

%\subsection{What kind of functions are in this function space?}

%\NOTE{This section should be easy reading, needs more work. We can even include a plot of some nice function, and maybe one of a quite wild function and its Walsh approximation.}

It is of course important to have an understanding of which functions exactly are in such a Walsh space $\Wspace$ with smoothness parameter $\alpha$. This analysis has been done in \cite{Dic2007,Dic2008,Dic2009}.
Classically, one is interested in (smooth) functions $f : [0,1]^s \to \R$ for which all mixed partial derivatives up to order $\alpha$ in each variable are square integrable.
This is a Sobolev space of smoothness $\alpha$ which is often considered for this type of problems.
In \cite{Dic2008, Dic2009} a continuous embedding of certain Sobolev spaces into $\Wspace$ was shown. Consequently, the results we are going to establish in the following for functions in $\Wspace$ also apply automatically to what we normally consider as ``smooth'' functions, for instance, functions which have square integrable partial mixed derivatives up to order $\alpha$ in each variable. One of the simplest type of functions in this space are multivariate polynomials which make up nice testing examples for computer implementations.

\subsection{Higher order digital nets}\label{sec:digital-nets}

Higher order digital nets were introduced in \cite{Dic2008}. Higher order polynomial lattice point sets, which are the focal point of this paper, are a special class of higher order digital nets. For that reason and since we will compare the explicit construction for higher order digital nets from \cite{Dic2008} with the construction given in this paper, we will review the necessary details here. For more information we refer to \cite[Chapter~15]{DP2010}.

For a prime number $b$ we always identify $\F_b$, the finite field with $b$ elements, with $\Z_b=\{0,\ldots,b-1\}$ endowed with the usual arithmetic operations modulo $b$.

First we define higher order digital nets using the digital construction scheme.
As we will need to be able to identify integers with vectors over a finite field by using its base $b$ representation, and then later have to be able to consider vectors of integers as well, we will denote a vector over a finite field $\F_b$ by $\vecc{h}$, in contrast to vectors over $\Z$ or $\R$, which will be denoted  by $\vec{h}$.
\begin{definition}[Digital construction scheme of a digital net over $\F_b$]  Let $b$ be a prime and let $n,m,s \geq 1$ be integers, where $n \geq m$. Let $C_1,\dots,C_s$ be $n \times m$ matrices over the finite field $\F_b$ of order $b$. Now we construct $b^m$ points in $[0,1)^s$: for $0 \le h < b^m$, identify each $h = \sum_{i=0}^{m-1} h_i \, b^i$ with a vector over the finite field
  \begin{align*}
    \vecc{h}
    &:=
    ( h_0, \ldots, h_{m-1})^\transp \in \F_b^m
    .
  \end{align*}
For $1 \le j \le s$ multiply the matrix $C_j$ by $\vecc{h}$ using arithmetic over $\mathbb{F}_b$ to obtain a vector $\vecc{y}_{h,j} \in \F_b^n$:
  \begin{align}\label{eq:mv}
    C_j \, \vecc{h}
    &
    =:
    \vecc{y}_{h,j}
    =
    ( y_{h,j,1}, \ldots, y_{h,j,n} )^\transp
    \in \F_b^n
    ,
  \end{align}
from which the $h$th point $\vec{x}_h$ of the digital net is found by interpreting the coordinates of $\vecc{y}_{h,j}$ as the base $b$ digits of $x_{h,j}$:
  \begin{align*}
    x_{h,j}
    &:=
    \sum_{i=1}^n y_{h,j,i} \, b^{-i}
    \in [0,1) .
  \end{align*}
Now set $\vec{x}_h = ( x_{h,1}, \ldots, x_{h,s} )^\transp \in [0,1)^s$ to be the $h$th point. The set $\{\vec{x}_0, \ldots, \vec{x}_{b^m-1}\}$ is called a \emph{digital net} over $\F_b$ with generating matrices $C_1, \ldots, C_s$.
\end{definition}

%\NOTE{We don't need to define what $(t, \alpha, \beta, n \times m, s)$-nets over $\Z_b$ are, I think. Even better would be if we can just talk about $(t, \alpha, n \times m, s)$-nets with $n=\alpha m$.}

%In what follows we take the bijection $\varphi$ to be the identity mapping and we only consider prime $b$.
This definition of a digital net generalizes the classical construction scheme, e.g., \cite{Nie92}, on which classical digital $(t,m,s)$-nets are based upon, by allowing for generating matrices which are not necessarily square.
%As there are linear independence conditions for the generating matrices in the classical setting, so are there linear independence conditions for the generating matrices of higher order digital nets.
%Such higher order digital nets are denoted as digital $(t,\alpha,\beta,n \times m,s)$-nets over $\F_b$, where $\alpha$ denotes the smoothness of the integrand, and, for a given point set, $t$~and $\beta$ can be considered functions of the unknown integrand smoothness $\alpha$.
%In our analysis here, we consider the smoothness $\alpha$ to be given and fixed.
The generating matrices $C_j$ are of size $n \times m$, and so, the number of rows $n$ determines the resolution at which the points of the net are placed in the unit cube, i.e., all base $b$ digits after position~$n$ are zero. The integration error then behaves like $O(b^{-k} k^{\alpha s})$, see \cite{Dic2008}, where $\alpha$ is the smoothness of the integrand and $k$ is the \emph{strength} of the net (in accordance with the respective property of classical nets). The strength of the net is defined via linear independence properties of the rows of the generating matrices, see \cite{Dic2007,Dic2008}. For higher order nets one can achieve $k \approx \min(\alpha m, n)$ and hence, provided that $n \ge \alpha m$, one obtains a convergence order of $b^{-\alpha m} (\alpha m)^{\alpha s} \asymp N^{-\alpha} (\log N)^{\alpha s}$, where $N = b^m$ is the number of quadrature points.

We now explain the explicit construction of a higher order digital net in $s$ dimensions for a maximum smoothness $d$ as described in \cite{Dic2008}. The explicit construction starts from a given $(t',m,sd)$-net in base $b$, that is, a classical digital net in $sd$ dimensions for which the generating matrices $C_1, \ldots, C_{sd} \in \F_b^{m \times m}$ are known.

From these $sd$ given matrices, $s$ new generating matrices $C_j^{(d)}$ are constructed of size $dm \times m$ by vertically stacking the first rows from the group of $d$ consecutive matrices $C_{(j-1)d+1},\ldots,C_{jd}$, then the second rows of the same $d$ matrices and so on, until all $dm$ rows have been stacked. More precisely, let $C_j = (\boldsymbol{c}_{j,1}^\top \ldots \boldsymbol{c}_{j,m}^\top)^\top$, where $\boldsymbol{c}_{j,k}^\top$ denotes the $k$th row of the matrix $C_j$. Then $C_j^{(d)} = (\boldsymbol{c}_{(j-1)d+1, 1}^\top \ldots \boldsymbol{c}_{jd,1}^\top \ldots \boldsymbol{c}^\top_{(j-1)d+1,m} \ldots \boldsymbol{c}^\top_{jd, m})^\top$. For more information on these higher order digital nets we refer the reader to \cite{Dic2008} and \cite[Chapter~15]{DP2010}.

An important concept for the error analysis in the next section
%, see\RefEq{eq:error-f}, 
is the \emph{dual net}.
It defines the set of Walsh coefficients which are not integrated exactly by the digital net and can therefore be used to write down the integration error.
%\NOTE{Classically the dual includes $\vec{0}$, Josef almost always defines it without, should decide what to use and modify accordingly. In the beginning of the new book they include zero, at the end they don't.}
\begin{definition}[Dual net]\label{def_dual}
For a digital net over $\F_b$ with generating matrices $C_1, \ldots, C_s \in \F_b^{n \times m}$ we define its \emph{dual net} by
  \begin{align*}
    \Dual(C_1,\ldots,C_s)
    &:=
\left\{ \vec{k} \in \N_0^s : C_1^\transp \, \vecc{k}_1 + \cdots + C_s^\transp \, \vecc{k}_s = \vecc{0} \right\}
    ,
  \end{align*}
where for a scalar component $k = \sum_{i=0}^\infty \kappa_i \, b^i$ in $\vec{k}$ we define an associated vector over the finite field $\vecc{k} = ( \kappa_0, \ldots, \kappa_n )^\transp \in \F_b^n$.
  %\NOTE{Don't we need to consider $\varphi$ here? Even if $b$ is prime?}
\end{definition}

\subsection{The worst-case error}

We define the worst-case error of numerical integration using a cubature rule $Q$ for functions in a Banach space $\mathcal{F}$ by
\begin{align*}
  e(Q, \mathcal{F})
  &:=
  \sup_{\substack{f \in \mathcal{F} \\ \|f\|_{\mathcal{F}} \le 1}} | I(f) - Q(f) |
  .
\end{align*}
We now assume the cubature rule $Q$ to be a quasi-Monte Carlo rule\RefEq{eq:qmc-rule} using a (higher order) digital net as its node set and denote it by $\Qnet$.
For any $f$ having an absolutely convergent Walsh series representation we can write the integration error for $\Qnet$ as a sum over the dual net to obtain
\begin{align}\label{eq:error-f}
  | I(f) - \Qnet(f) |
  &=
  \left| \sum_{\vec{0} \ne \vec{k} \in \Dual} \widehat{f}_{\vec{k}} \right|
  \le
  \sum_{\vec{0} \ne \vec{k} \in \Dual} |\widehat{f}_{\vec{k}}|
  .
\end{align}
%Here the sum is taken over all non-zero $\vec{k}$ in the dual of the digital net, $\Dual \subseteq \N_0^s$, which is dependent on $\Qnet$. 
For $f \in \Wspace$ we can now use\RefEq{eq:bounded-Walsh-coeff} to obtain
\begin{align}\label{eq:error-bound-for-f}
  | I(f) - \Qnet(f) |
  &\le
  \| f \|_{\Wspace} \, \sum_{\vec{0} \ne \vec{k} \in \Dual} \ralpha
  .
\end{align}
Since we can obtain equality for a worst-case function $\zeta \in \Wspace$ having Walsh series representation
\begin{align*}
  \zeta(\vec{x})
  &=
  \sum_{\vec{k} \in \N_0^s} \ralpha \, \wal_{\vec{k}}(\vec{x})
  ,
\end{align*}
we find the following expression for the worst-case error in $\Wspace$ for a quasi-Monte Carlo rule $\Qnet$ based on a higher order digital net in base $b$:
\begin{align}\label{eq:worst-case-error}
  e(\Qnet, \Wspace)
  &=
  \sum_{\vec{0} \ne \vec{k} \in \Dual} \ralpha
  .
\end{align}
The cubature rules in this paper will be constructed in such a way that they have a worst-case error which is near optimal for the given function space $\Wspace$.
For a given value of $\alpha > 1$ the worst-case error behaves like $O(N^{-\alpha} (\log N)^{\alpha s})$ for $N$ integration nodes (see \cite{Dic2008}) which is essentially best possible according to a lower bound from {\v{S}}arygin~\cite{Sar63}.

\subsection{Higher order polynomial lattice rules}

%Multiplication (and division) of polynomials can be seen as convolution\NOTE{\ldots\ I want to make the connection using the representation of multiplication with a fixed polynomial being a matrix-vector product.}

In \cite{DP2007} the classical polynomial lattice rules \cite{Nie92} were generalized to form higher order polynomial lattice rules.
Just like classical polynomial lattice point sets are a special class of digital nets, higher order polynomial lattice point sets are a special class of higher order digital nets. For simplicity, we define the (higher order) polynomial lattice rules over a finite field  $\F_b$ of prime order $b$ only. The main object for the construction of polynomial lattice rules are formal Laurent series, i.e., expressions of the form $\sum_{i=\ell}^\infty w_i \, \x^{-i}$, where $\ell \in \Z$ and $w_i \in \F_b$. We denote the set of formal Laurent series by $\F_b((\x^{-1}))$. These Laurent series then need to be mapped to integration nodes over the unit interval $[0, 1)$. Define the map $v_n : \F_b((\x^{-1})) \to [0,1)$ by
\begin{align}\label{eq:vn}
%   v_n\left( \sum_{i=\ell}^\infty w_i \, \x^{-i} \right)
%   &:=
%   \sum_{i=\max(1,\ell)}^n w_i \, b^{-i}
   v_n\left( \sum_{i=\ell}^\infty w_i \, \x^{-i} \right)
   &:=
   \sum_{i=\max(\ell,1)}^n w_i \, b^{-i}.
\end{align}
Similar to the case for digital nets, we now need to be able to identify an integer $h$ with a polynomial in $\F_b[\x]$ by considering $h$ in its base $b$ representation, the associated polynomial will be denoted by $h(\x)$. The details are given in the following definition.

\begin{definition}[Polynomial lattice rule]\label{def:poly-lattice-rule}
  Let $b$ be prime and $1 \le m \le n$.
  For a given dimension $s \ge 1$, choose $p(\x) \in \F_b[\x]$ with $\deg(p) = n \ge 1$ and let $q_1(\x),\ldots,q_s(\x) \in \F_b[\x]$.
  Now we construct $b^m$ points in $[0,1)^s$: for $0 \le h < b^m$, identify each $h = \sum_{i=0}^{m-1} h_i \, b^i$ with a polynomial over $\F_b$
  \begin{align*}
    h(\x)
    &:=
    \sum_{i=0}^{m-1} h_i \, \x^i \in \F_b[\x] .
  \end{align*}
  Then the $h$th point is obtained by setting
  \begin{align*}
    \vec{x}_h
    &:=
    \left( v_n\left( \frac{h(\x) \, q_1(\x)}{p(\x)} \right) , \ldots , v_n\left( \frac{h(\x) \, q_s(\x)}{p(\x)} \right) \right) \in [0,1)^s .
  \end{align*}
  A quasi-Monte Carlo rule using this point set is called a \emph{polynomial lattice rule}.
\end{definition}

One obtains classical polynomial lattice rules from \RefDef{def:poly-lattice-rule} by taking $n = m$.
%Furthermore, we remark that for our results only the degree of the polynomial $p(x)$ is important and not the specific choice of $p(\x)$ itself.
For simplicity we will assume that $p(\x)$ is irreducible over $\F_b$, though this assumption could be removed by a more intricate analysis. 
%Further, for algorithmic simplicity, we will assume that $p(\x)$ is a primitive polynomial over $\F_b$, in that case $g(\x) = \x$ is a generator for the cyclic group $(\F_b[\x]/p(\x))^\mg = (\F_b[\x]/p(\x))\setminus \{0\}$ which will come of use in the construction algorithm.
We define
\begin{align*}
  G_{b,n}
  &=
  \{ v(\x) \in \F_b[\x] \setminus \{ 0 \}: \deg(v) < n \},
\end{align*}
which will be the set from which we will select the generating polynomials $q_j(\x)$.
Clearly, as $p(\x)$ is irreducible and $\deg(p) = n$, this equals the multiplicative group
\begin{align*}
  G_{b,n}
  &=
  (\F_b[\x]/p(\x))^\mg = \{ g(\x)^\beta : 0 \le \beta < b^n-1 \}
  ,
\end{align*}
where $g(\x)$ is a generator for the multiplicative
group $(\F_b[\x]/p(\x))^\mg$, e.g., we can take $g(\x) = \x$ when $p(\x)$ is primitive.

Since a polynomial lattice point set is a special case of a digital net, we can find the generating matrices $C_1,\ldots,C_s \in \F_b^{n \times m}$ from the generating vector $\vec{q}(\x) = (q_1(\x),\ldots,q_s(\x))$.
For $1 \le j \le s$ consider the Laurent expansions
\begin{align*}
  \frac{q_j(\x)}{p(\x)}
  &=
  \sum_{i=\ell_j}^\infty u_i^{(j)} \, \x^{-i} \in \F_b((\x^{-1}))
  .
\end{align*}
Then the elements $c_{k,\ell}^{(j)}$ of the $n \times m$ generating matrix $C_j$ over $\F_b$ are given by
\begin{align}\label{eq:generatingmatrices}
  c_{k,\ell}^{(j)}
  &=
  u_{k+\ell}^{(j)}
  ,
\end{align}
for $1 \le k \le n$, $0 \le \ell \le m - 1$; see, e.g., \cite[Section~10.1]{DP2010}.

%\NOTE{Rehash the further part of Jan's section on polynomial lattice rules.}

In\RefEq{eq:worst-case-error} we used the dual net to obtain the worst-case error.
In the case of a polynomial lattice rule the dual is given in the next definition.
(We use the convention $\deg(0) = -\infty$.)
\begin{definition}[Dual polynomial lattice]\label{def_dual_poly}
  A polynomial lattice with generating vector $\vec{q}(\x) = (q_1(\x),\ldots,q_s(\x)) \in (\F_b[\x])^s$ modulo $p(\x) \in \F_b[\x]$ has a \emph{dual polynomial lattice}
  \begin{align*}
    \Dual(\vec{q}(\x), p(\x))
    &:=
    \left\{
      \vec{k} \in \N_0^s : \sum_{j=1}^s k_j(\x) \, q_j(\x) \equiv a(\x) \pmod{p(\x)}
        \quad \text{with } \deg(a) < n - m
    \right\}
    .
  \end{align*}
\end{definition}
%\NOTE{Isn't then the dual net definition from above incorrect? Shouldn't there be part of the zero vector to be allowed non-zero? --- I think I verified it to be correct, but can't remember how it works.}
A proof for the equivalence of Definition~\ref{def_dual} and Definition~\ref{def_dual_poly} for polynomial lattices follows from \cite[Lemma~15.25]{DP2010}.

Specifically for a polynomial lattice rule with generating vector $\vec{q}(\x) = (q_1(\x),\ldots,q_s(\x))$ modulo $p(\x)$ having $b^m$ points, it follows from\RefEq{eq:worst-case-error} that its worst-case error in $\Wspace$ satisfies
\begin{align}\label{eq:poly-worst-case-error}
  e_{b^m,\alpha}(\vec{q}(\x), p(\x))
  &=
  \sum_{\vec{0} \ne \vec{k} \in \Dual}
  r_\alpha(\vec{\gamma},\vec{k})
  ,
\end{align}
with $\Dual$ the dual polynomial lattice.

\subsection{The component-by-component construction of higher order polynomial lattice rules}\label{sec:cbc}

The component-by-component construction algorithm was introduced by Korobov~\cite{Kor59}, see also \cite[Theorem~18, p.~120]{Kor63}, and later re-invented in \cite{SR2002} to construct the generating vector of an integer lattice rule.
This algorithm first finds the optimal one-dimensional generating vector, which is subsequently extended in an optimal way to a two-dimensional generating vector and so on.
\RefAlg{alg:cbc} spells out the details for the construction in the case of higher order polynomial lattice rules.

\begin{algorithm}[H]
\caption{General form of CBC construction of higher order polynomial lattice rules}\label{alg:cbc}
\begin{algorithmic}
\STATE \textbf{Input:} base $b$ a prime, number of dimensions $s$, number of points $b^m$, smoothness $\alpha > 1$, and weights $\vec{\gamma} = (\gamma_j)_{j\ge1}$
\STATE \textbf{Output:} Generating vector $\vec{q}(\x) = (q_1(\x), \ldots, q_s(\x)) \in G_{b,n}^{s}$ \\[2mm]
\STATE Choose an irreducible polynomial $p(\x) \in \F_b[\x]$, with $\deg(p)=n$ and $n=\alpha m$
\FOR{$d=1$\TO $s$}
  \STATE Set $q_d(\x) \in G_{b,n}$ by minimizing $e_{b^m,\alpha}((q_1(\x),\dots,q_d(\x)),p(\x))$ as a function of $q_d(\x)$
\ENDFOR
\RETURN $\vec{q}(\x)=(q_1(\x),\dots,q_s(\x))$
\end{algorithmic}
\end{algorithm}

The analysis of the component-by-component algorithm adjusted to the case of higher order polynomial lattice rule was done in \cite{BDGP2011}.
The following theorem shows that \RefAlg{alg:cbc} achieves almost optimal rates of convergence.
%\NOTE{The previous sentence doesn't really say what we want, funny though.}
For a proof we refer to \cite{BDGP2011}. %\NOTE{Check that $G_{b,n}$ also includes $0$.}
\begin{theorem}
  Let $b$ be prime, $s,n \in \N$ and $p(\x) \in \F_b[\x]$ be irreducible with $\deg(p) = n$, $\alpha > 1$.
  Suppose $(q_1(\x),\ldots,q_s(\x)) \in G_{b,n}^s$ is constructed using \RefAlg{alg:cbc}. Then for all $d=1,\ldots,s$ we have a bound on the worst-case error as follows:
  \begin{align*}
    e_{b^m,\alpha}((q_1(\x),\ldots,q_d(\x)), p(\x))
    &\le
    \frac{1}{b^{\min(\tau m, n)}}
    \prod_{j=1}^d \left( 1 + 3 \gamma_j^{1/\tau} C_{b,\alpha,\tau} \right)^\tau
    ,
    &
    \forall \; 1 \le \tau < \alpha
    ,
  \end{align*}
  where
  \begin{align*}
    C_{b,\alpha,\tau}
    &:=
    \frac{(b-1)^\alpha}{b^{\alpha/\tau}-b}
    \prod_{i=1}^{\alpha-1} \frac{1}{b^{i/\tau}-1}
    +
    \begin{cases}
      \; \alpha - 1 & \text{if } \tau = 1, \\[1mm]
      \displaystyle \frac{(b-1)((b-1)^{\alpha-1}-(b^{1/\tau}-1)^{\alpha-1})}{(b-b^{1/\tau})(b^{1/\tau}-1)^{\alpha-1}}  & \text{if } \tau > 1.
    \end{cases}
%    ,
%    \\
%    \tilde{C}_{b,\alpha,\tau}
%    &:=
%    \begin{cases}
%      \alpha - 1, & \text{if } \tau = 1, \\
%      \frac{(b-1)((b-1)^{\alpha-1}-(b^{1/\tau}-1)^{\alpha-1})}{(b-b^{1/\tau})(b^{1/\tau}-1)^{\alpha-1}} , & \text{if } \tau > 1.
%    \end{cases}
  \end{align*}
\end{theorem}

Formula\RefEq{eq:poly-worst-case-error} for the worst-case error is not in a usable form for computation due to the infinite sum. The next lemma shows how to obtain a closed-form expression which resembles the formula for the worst-case error as it appears when the space of integrands is a reproducing kernel Hilbert space, see \cite{BD2011}.
\begin{lemma}\label{lem:wce}
  The worst-case integration error in $\Wspace$, $\alpha > 1$, associated with a polynomial lattice rule with generating vector $\vec{q}(\x) = (q_1(\x),\ldots,q_s(\x))$ modulo $p(\x)$ having $b^m$ points satisfies
  \begin{align}\label{eq:poly-wce-prod}
    e_{b^m,\alpha}((q_1(\x),\ldots,q_s(\x)), p(\x))
    &=
    -1 + \frac{1}{b^m} \sum_{h=0}^{b^m-1} \prod_{j=1}^s (1 + \gamma_j \, \omega_\alpha(x_{h,j}))
    ,
  \end{align}
  where, using\RefEq{eq:ralpha},
  \begin{align}\label{eq:omega}
    \omega_\alpha(x)
    &:=
    \sum_{k=1}^\infty r_\alpha(k) \, \wal_k(x)
    .
  \end{align}
\end{lemma}
\begin{proof}
  We make use of the character property of digital nets (see \cite[Lemma~4.75]{DP2010}).
  When $\left\{ \vec{x}_h \right\}^{b^m-1}_{h=0}$ are the points of a digital net (or a polynomial lattice rule) and $\Dual$ is its dual net, then
  \begin{align*}
    \frac{1}{b^m}\sum_{h=0}^{b^m-1} \wal_{\vec{k}}(\vec{x}_h)
    &=
    \begin{cases}
      1 & \text{if } \vec{k} \in \Dual , \\
      0 & \text{otherwise} .
    \end{cases}
  \end{align*}
  Thus, starting from\RefEq{eq:poly-worst-case-error}, we obtain
  \begin{align*}
    e_{b^m,\alpha}((q_1(\x),\ldots,q_s(\x)), p(\x))
    &=
    \sum_{\vec{0} \ne \vec{k} \in \N_0^s}
    r_\alpha(\vec{\gamma},\vec{k})
    \frac{1}{b^m}
      \sum_{h=0}^{b^m-1} \wal_{\vec{k}}(\vec{x}_h)
    \\
    &=
    -1 + \frac{1}{b^m} \sum_{h=0}^{b^m-1} \sum_{\vec{k} \in \N_0^s} r_\alpha(\vec{\gamma},\vec{k}) \, \wal_{\vec{k}}(\vec{x}_h)
    \\
    &=
    -1 + \frac{1}{b^m} \sum_{h=0}^{b^m-1} \prod_{j=1}^s ( 1 + \gamma_j \, \omega_\alpha(x_{h,j}) )
    .
  \end{align*}
\end{proof}

As in \cite{NC2006-prime} we can now do a basic operation count for the computational cost of \RefAlg{alg:cbc}.
In comparison to \cite{NC2006-prime}, the analysis is a little bit more involved here as the evaluation of\RefEq{eq:poly-wce-prod} involves calculating
\begin{align}\label{eq:omega-complication}
  \omega_\alpha\left( v_n\left( \frac{h(\x) \, q(\x)}{p(\x)} \right) \right)
  && \text{for } h=0,\ldots,b^m-1 \text{ and all } q(\x) \in G_{b,n}
  ,
\end{align}
in each iteration.
Assuming $c_\omega = c_\omega(\alpha,n)$ to be the cost of evaluating $\omega_\alpha(x)$ and $c_v = c_v(n)$ the cost of mapping and calculating the Laurent expansion as well as calculation the polynomial product, the cost of a straightforward implementation of \RefAlg{alg:cbc} is $O(s^2 b^n b^m (c_\omega + c_v))$ where $n=\alpha m$.
However, calculating $\omega_\alpha(v_n(h(\x)\,q(\x)/p(\x)))$ efficiently for given $h(\x)$ and $q(\x)$ is an important issue which is solved in the next section and in practice it would be inefficient to calculate these values on the fly whenever needed.
To that end we model the algorithm into a more tangible form using the techniques from \cite{NC2006-prime,NC2006-kernels} to obtain a fast component-by-component algorithm that makes use of a circular convolution which can be calculated by means of fast Fourier transforms (FFTs).

\section{Fast construction of higher order polynomial lattice rules}\label{sec:fastcbc}

The exposition here mainly follows the techniques from \cite{NC2006-prime,NC2006-kernels}, but, as mentioned in the previous section, the analysis is more complicated due to the need to calculate\RefEq{eq:omega-complication}.
The derivation of the fast algorithm is kept concise by relying as quickly as possible on the structure of the underlying multiplicative group, but we need to take into consideration the cost $c_v$ of working with polynomials over finite fields.
In \RefSec{sec:kernel} we will give efficient methods to calculate $\omega_\alpha$.

The product over $j$ in Lemma~\ref{lem:wce} can be reused and extended from the previous iteration.
We store this product in a vector $\vec{P}_d = (P_d(0),\ldots, P_d(b^m-1))$ of length $b^m$, where%and find the product vector for dimension~$d$ by
\begin{align*}
  P_d(h)
  :=
  \prod_{j=1}^d ( 1 + \gamma_j \, \omega_\alpha(x_{h,j}) )
%  &=
%  P_{d-1}(h) \, ( 1 + \gamma_j \, \omega_\alpha(x_{h,d}) )
%  \\
  &=
  P_{d-1}(h) \, \left( 1 + \gamma_j \, \omega_\alpha\left( v_n\left( \frac{h(\x) \, q_d(\x)}{p(\x)} \right) \right) \right)
  ,
\end{align*}
for all $0 \le h < b^m$ and $P_0(h) = 1$. Thus $\vec{P}_d$ can be calculated using the stored value for $\vec{P}_{d-1}$. Hereby we reduce the construction cost by a factor of $s$ at the cost of $O(b^m)$ memory.

The computations of $\omega_\alpha( v_n(h(\x) \, q(\x) / p(\x) ) )$ could be done in the initialization of the algorithm.
Since $v_n$ only depends on the negative powers of $\x$ we in fact have
\begin{align*}
  v_n\left( \frac{h(\x) \, q(\x)}{p(\x)} \right)
  &=
  v_n\left( \frac{h(\x) \, q(\x)}{p(\x)} \bmod{1(\x)} \right)
  =
  v_n\left( \frac{h(\x) \, q(\x) \bmod{p(\x)}}{p(\x)} \right)
  .
\end{align*}
So, for fixed $p(\x)$, we can think of $v_n$ as being a function from $\F_b[\x]/p(\x) = \{ w(\x) \in \F_b[\x] : \deg(w) < n \}$ to $[0,1)$.
We can precompute these $b^n$ values giving a construction cost of $O(s b^n b^m c_v + b^n (c_\omega + c_v))$ at a cost of $O(b^m+b^n)$ memory.
However, the cost $c_v$ is presumably dominating $c_\omega$ (most certainly so for the $\omega_\alpha$ expressions we will derive in \RefCol{col:base2}).
It is standard practice to use a lookup table based on a generator when doing multiplications over a finite field.
Making this change the construction cost becomes $O(s b^n b^m + b^n (c_\omega + \tilde{c}_v))$ at a cost of $O(b^m + 2 b^n)$ memory (we have explicitly written the constant for clarity: there is a $O(b^n)$ cost for the values of $\omega_\alpha$ and a $O(b^n)$ cost for the lookup table).
Here, $\tilde{c}_v$ is a lot cheaper than $c_v$ as one has to multiply only by the same generator to construct the table.
We do note however that the $O(b^n)$ memory cost grows exponentially with $\alpha$ as $n = \alpha m$.

For the lookup table we made use of the fact that there exists a generator $g(\x)$ for the multiplicative group for which
\begin{align*}
  (\F_b[\x]/p(\x))^\mg
  &:=
  \{ g(\x)^\beta \bmod{p(\x)}: 0 \le \beta < b^n-1 \}
  =
  \F_b[\x]/(p(\x)) \setminus \{ 0 \},
\end{align*}
since we assumed $p(\x)$ to be irreducible over $\F_b[\x]$.
For brevity we define the auxiliary function $\omega$ to make use of the indices w.r.t.\ the generator $g(\x)$:
\begin{align*}
  \omega : \Z_{b^n - 1} \to [0,1) : 
  \omega(\beta - \delta)
  &\hphantom{:}=
  \omega(\beta - \delta \bmod (b^n-1))
  \\
  &:= \omega_\alpha\left( v_n\left( \frac{h(\x) \, q(\x) \bmod{p(\x)}}{p(\x)} \right) \right)
  \\
  &\hphantom{:}=
  \omega_\alpha\left( v_n\left( \frac{g(\x)^\beta \, g(\x)^{-\delta} \bmod{p(\x)}}{p(\x)} \right) \right)
  \\
  &\hphantom{:}=
  \omega_\alpha\left( v_n\left( \frac{g(\x)^{\beta-\delta} \bmod{p(\x)}}{p(\x)} \right) \right)
  ,
\end{align*}
where $h(\x)$ and $q(\x)$ are such that $h(\x) = g(\x)^\beta \bmod{p(\x)}$ and $q(\x) = g(\x)^{-\delta} \bmod{p(\x)}$.

Now consider the worst-case error explicitly in terms of $q_d(\x)$ as $E_d(q_d(\x))$. Then we can write the worst-case error iteratively in the form
\begin{align*}
  E_d(q_d(\x))
  &:=
  e_{b^m,\alpha}((q_1(\x),\ldots,q_d(\x)),p(\x))
  \\
  &\hphantom{:}=
  -1 + \frac{1}{b^m} \sum_{h=0}^{b^m-1} P_{d}(h)
  \\
  &\hphantom{:}=
  -1 + \frac{1}{b^m} \sum_{h=0}^{b^m-1} P_{d-1}(h) + \frac{\gamma_d}{b^m} \sum_{h=0}^{b^m-1} P_{d-1}(h) \, \omega_\alpha( v_n( h(\x) \, q_d(\x) / p(\x) ) )
  \\
  &\hphantom{:}=
  e_{b^m,\alpha}((q_1(\x),\ldots,q_{d-1}(\x)),p(\x)) + \frac{\gamma_d}{b^m} \sum_{h=0}^{b^m-1} P_{d-1}(h) \, \omega_\alpha( v_n( h(\x) \, q_d(\x) / p(\x) ) )
  \\
  &\hphantom{:}=
  e_{b^m,\alpha}((q_1(\x),\ldots,q_{d-1}(\x)),p(\x)) 
  \\
  &\qquad + \frac{\gamma_d}{b^m} \omega_\alpha(0) + \frac{\gamma_d}{b^m} \sum_{h=1}^{b^m-1} P_{d-1}(h) \, \omega_\alpha( v_n( h(\x) \, q_d(\x) / p(\x) ) )
  ,
\end{align*}
where the worst-case error for the zero-dimensional rule is~$0$.
The main computational burden is now hidden in calculating the last sum which we can write in terms of the auxiliary function $\omega$ as an \emph{extended sum}:
\begin{align}\label{eq:extendedsum}
  \sum_{h=1}^{b^m-1} \omega_\alpha( v_n( h(\x) \, q_d(\x) / p(\x) ) ) \, P_{d-1}(h)
  &=
  \sum_{\beta=0}^{b^n-2} \omega(\beta-\delta) \, Q_{d-1}( \beta )
\end{align}
where $\delta$ is such that $q_d(\x) = g(\x)^{-\delta} \bmod{p(\x)}$ and
\begin{align*}
  Q_{d-1}(\beta) 
  &:=
  \begin{cases}
    P_{d-1}(g(\x)^\beta \bmod{p(\x)}) & \text{if } \deg(g(\x)^\beta \bmod{p(\x)}) < m, \\
    0 & \text{otherwise}.
  \end{cases}
\end{align*}
Let $\vec{Q}_d = (Q_d(0),\ldots, Q_d(b^{n}-2))$, then we have $\vec{Q}_d = \Pi^\transp_{g(\x)^{-1}} \vec{P}_d$, where
\begin{equation*}
\Pi^\transp_{g(\x)^{-1}} = \left(a_{u,v}\right)_{0 \le u \le b^n-2, 1 \le v < b^m}
\end{equation*}
and
\begin{equation*}
a_{u,v} = \left\{\begin{array}{ll} 1 & v(\x) \equiv g(\x)^u \bmod{p(\x)}, \\ 0 & \mbox{otherwise}. \end{array} \right.
\end{equation*}
Thus $\vec{Q}_d$ is obtained by permuting the elements of the vector $(\vec{P}_d, \vec{0}) \in \mathbb{R}^{b^n-1}$.

This extended sum\RefEq{eq:extendedsum}, calculated for all possible
choices of $q_d(\x) = g(\x)^{-\delta} \bmod{p(\x)} \in (\F_b[\x]/p(\x))^\mg$, i.e., $0 \le \delta < b^n-1$, is in fact a \emph{circular convolution} of length $b^n-1$
\begin{align}
  \nonumber
  S_d(\delta)
  &:=
  \sum_{\beta=0}^{b^n-2} \omega(\beta-\delta \bmod{(b^n-1)}) \, Q_{d-1}(\beta)
  \\
  \label{eq:cconv}
  &\hphantom{:}=
  \sum_{\beta=0}^{b^n-2} \omega(\beta) \, Q_{d-1}(\beta+\delta \bmod{(b^n-1)})
  .
\end{align}
Calculating this convolution in the Fourier domain by the use of fast Fourier transforms (FFTs) takes time $O(b^n \log b^n)$, see \cite{Nus82} for a general reference. We obtain a construction cost for the fast component-by-component algorithm using FFTs of $O(s b^n \log b^n + b^n ( c_\omega + \tilde{c}_v ))$ using $O(b^n)$ memory. In other words, as $n=\alpha m$, the factor $b^m$ in the original complexity has been reduced to $\alpha \log b^m$.
Asymptotically this is always faster for increasing $m$.
%, but furthermore, it would be impractical to obtain stored values of $\omega_\alpha(x_{h,j})$ in \RefAlg{alg:cbc} without having any structure to hold on to, i.e., one would have to execute a polynomial multiplication every single time (which would be quite costly).

We end this section with an overview of the complexities and their memory trade of:
\begin{center}\begin{tabular}{c|cc}
            & Construction cost &  \\
  Algorithm & $= s \text{\{iteration cost\}} + \text{\{initialization cost\}}$ & Memory cost \\\hline 
  Straightforward & $s^2 b^n b^m ( c_\omega + c_v )$ &  $\vphantom{A^{A^{A}}}$\\
  Cache $\vec{P}_d$ vector & $s b^n b^m ( c_\omega + c_v )$ & $b^m$ \\
  Precalculate $\omega$ & $s b^n b^m + b^n ( c_\omega + \tilde{c}_v )$ & $b^n$ \\
  Fast convolution & $s b^n \log b^n + b^n ( c_\omega + \tilde{c}_v )$ & $b^n$ \\
\end{tabular}\end{center}
All these algorithms, except the fast convolution algorithm, have iteration times of $O(b^n b^m)$ and so they will all be asymptotically slower than the fast convolution algorithm.
Timings on a real machine for $b=2$ show the break even point to be at $m=5$ for $\alpha=2$ and $m=6$ for $\alpha=3$.
The fast component-by-component algorithm based on fast convolution is given in \RefAlg{alg:fast-cbc}.

\begin{algorithm}[H]
\caption{Fast CBC construction of higher order polynomial lattice rules}\label{alg:fast-cbc}
\begin{algorithmic}
\STATE \textbf{Input:} base $b$ a prime, number of dimensions $s$, number of points $b^m$, smoothness $\alpha > 1$, and weights $\vec{\gamma} = (\gamma_j)_{j\ge1}$
\STATE \textbf{Output:} Generating vector $\vec{q}(\x) = (q_1(\x), \ldots, q_s(\x)) \in G_{b,n}^{s}$ \\[2mm]
\STATE Choose an irreducible polynomial $p(\x) \in \F_b[\x]$,
$\deg(p)=n$ and $n=\alpha m$ and generator $g(\x)$
\STATE Set $e_0 = 0$,
       $\displaystyle \vec{Q}_0 = \Pi_{g(\x)^{-1}}^\transp \! \begin{pmatrix} \vec{1}_{b^m \times 1} \\ \vec{0}_{(b^n-b^m) \times 1} \end{pmatrix}$
       \\ and
       $\vec{\omega} = \Bigl( \omega_\alpha( v_n( g(\x)^{\delta} \pmod{p(\x)} / p(\x)) ) \Bigr)_{\delta=0,\ldots,b^n-2}^{\vphantom{A^A}}$
\FOR{$d=1$ \TO $s \vphantom{A^{A^A}}$}
  \STATE $\vec{S}_d = \vec{\omega} \circledast \vec{Q}_{d-1}$
  \quad (by (fast) circular convolution)
  \STATE $\displaystyle \delta = \argmin_{0 \le \delta < b^n-1} S_d(\delta)$
  \STATE Set $q_d(\x) = g(\x)^{\delta} \pmod{p(\x)}$
  \STATE Update/set $\vec{Q}_d$ and
  $\displaystyle e_d = e_{d-1} + \frac{\gamma_d}{b^m} \, \omega_\alpha(0) + \frac{\gamma_d}{b^m} \, S_d(\delta)$
\ENDFOR
\RETURN $\vec{q}(\x)=(q_1(\x),\dots,q_s(\x))$
\end{algorithmic}
\end{algorithm}

\section{Calculation of the worst-case error}\label{sec:kernel}

In this section we show how to calculate the infinite sum\RefEq{eq:omega} which appears in the worst-case error formula from \RefLem{lem:wce}. In \RefThm{thm:ndigits} we show that if $x$ can be represented exactly with $n$~digit precision in base~$b$, then $\omega_\alpha(x)$ can be computed in $O(\alpha n)$ operations. Following that, in \RefSec{sec:omega-alpha-computation}, \RefThm{thm:arbitrary-x} will state explicit forms for general $x \in [0,1)$. More importantly, for $b=2$, \RefCol{col:base2} gives explicit forms to compute $\omega_\alpha(x)$ exactly for arbitrary~$x$ and $\alpha=2$ and~$3$ using elementary computer operations.

\subsection{Technical definitions}

Before we can show how to compute $\omega_\alpha(x)$ we need to introduce some technical notation which will be used in the proofs in the next section. %, \RefSec{sec:omega-alpha-computation}.
%The reader who is only interested in the results can safely skip this exposition.
To motivate the notation we first look at $\omega_\alpha$ after expanding the definitions of $\wal_k(x)$ and $r_\alpha(k)$, see\RefEq{eq:walsh-1d} and\RefEq{eq:ralpha}, and using the non-zero digit expansion,\RefEq{eq:k-base-b-nz}, of $k = \sum_{i=1}^{\nzd{k}} \kappa_{a_i} \, b^{a_i}$, where all $\kappa_{a_i} \ne 0$:
\begin{multline}\label{eq:omega-step1}
  \omega_\alpha(x)
  =
  \sum_{k=1}^\infty r_\alpha(k) \, \wal_k(x)
  \\=
  \sum_{\substack{k=1 \\[0.5mm]
                  k = \sum_{i=1}^{\nzd{k}} \kappa_{a_i} \, b^{a_i}
                 }
  }^\infty
    \prod_{i=1}^{\min(\alpha,\nzd{k})} b^{-(a_i+1)} \, \exp(\twopii \kappa_{a_i} \xi_{a_i+1} / b)
    \;
    \underbrace{
    \prod_{i=\alpha+1}^{\nzd{k}} \exp(\twopii \kappa_{a_i} \xi_{a_i+1} / b)
    }_{\wal_{k'}(x)}
  .
\end{multline}
Due to definition\RefEq{eq:k-base-b-nz} we have $a_1 > \cdots > a_{\nzd{k}} \ge 0$, i.e., $\kappa_{a_1}$ is the most significant base~$b$ digit of~$k$, etc.
The second product in\RefEq{eq:omega-step1} can be seen as $\wal_{k'}(x)$ where $k'$ is defined by $k' = k- \sum_{i=1}^{\min(\alpha,\nzd{k})} \kappa_{a_i} \, b^{a_i}$ and thus $0 \le k' < b^{a_{\min(\alpha,\#k)}}$. Now observe that the sum over all $k \ge 1$ can be expanded into multiple sums over all possible digit expansions for all $k$'s with $r$ digits for $r \ge 1$.
That is, $\nzd{k}$ sums for the $a_i$ together with companioning sums for the $\kappa_{a_i}$ from~$1$ to~$b-1$, i.e.,
\begin{align}\label{eq:expand}
  \sum_{\substack{k=1 \\[0.5mm]
                  k = \sum_{i=1}^{\nzd{k}} \kappa_{a_i} \, b^{a_i}
                 }
  }^\infty
  G(k, x)
  &=
  \sum_{r=1}^\infty
  \;
  \underbrace{
  \sum_{a_1 = r-1 \vphantom{\kappa_{a_1}}}^\infty
  \cdots
  \sum_{a_r = 0}^{a_{r-1}-1}
  }_{\substack{r \text{ sums s.t.}\\[1mm] \infty > a_1 > \cdots > a_r \ge 0}}
  \;
  \underbrace{
  \;
  \sum_{\kappa_{a_1}=1}^{b-1}
  \;\cdots\;
  \sum_{\kappa_{a_r}=1}^{b-1}
  \;
  }_{r \text{ independent sums}}
  \;
  G\!\left( \sum_{i=1}^r \kappa_{a_i} \, b^{a_i}, x \right),
\end{align}
where $G(k,x) = r_\alpha(k) \wal_k(x)$.

To simplify notation and stress the structure in what follows, we define the following \emph{triangular sum operator} which sums over all $M \ge a_1 > \cdots > a_r \ge m$:
\begin{align}\label{eq:sum-operator}
  T_m^M(r)( g )
  &:=
  \underbrace{
    \sum_{a_1=m+r-1}^M \sum_{a_2=m+r-2}^{a_1-1} \cdots \sum_{a_r=m}^{a_{r-1}-1}
  }_{\text{$r$ sums}}
  g(a_1, \ldots, a_r)
  ,
\end{align}
and formally set the zero index sum, i.e., no sums to be taken, to be the identity mapping,
\begin{align*}
  T_m^M(0)(g) &:= g .
\end{align*}
Define the \emph{concatenation} of two such operators as putting the sums next to each other:
\begin{align*}
  (T_{m'}^{M}(t) \, T_{m}^{M'}(r-t))(g)
  &:=
  \underbrace{
  \underbrace{
    \sum_{a_1=m'+t-1 \vphantom{a_{t+1}=m'+(r-t)-1} }^{M} \cdots \sum_{a_t=m' \vphantom{a_{t+1}=m'+(r-t)-1} }^{a_{t-1}-1}
  }_{\text{$t$ sums}}
  \;
  \underbrace{
    \sum_{a_{t+1}=m+(r-t)-1 \vphantom{a_{t+1}=m'+(r-t)-1} }^{M'} \cdots \sum_{a_r=m \vphantom{a_{t+1}=m'+(r-t)-1} }^{a_{r-1}-1}
  }_{\text{$(r-t)$ sums}}
  }_{\text{$r$ sums}}
  g(a_1, \ldots, a_r)
  ,
\end{align*}
i.e., having two independent ranges $M \ge a_1 > \cdots > a_t \ge m'$ and $M' \ge a_{t+1} > \cdots > a_r \ge m$. We remark that although these sums might look haggardly, the interpretation of the sum operator by their summation range is a natural way to reason about it as the following lemma shows.

\begin{lemma}\label{lem:split}
  For any $M \ge n \ge m$ we can split $T_m^M(r)$ into $r+1$ sets of two independent ranges:
  \begin{align*}
    T_m^M(r)
    &=
    \sum_{t=0}^r (T_{n+1}^M(t) \, T_m^n(r-t))
    .
  \end{align*}
\end{lemma}
\begin{proof}
  Applying $T_m^M(r)$ to a function $g(a_1,\ldots,a_r)$ can be interpreted combinatorially as having to distribute $r$ objects in $M-m+1$ different positions, numbered from $m$ to $M$, which can each hold at most one object and then accumulating the result of applying the function $g$ to this ensemble.
  It is trivial to note that we can split the range in two non-overlapping ranges and consider all partitions of $r$ to distribute the objects over the two ranges.
\end{proof}

Specifically we find the following expansions of this summation operator for $r=1,2,3$:
\begin{subequations}\label{eq:expansions}
\begin{align}
  \label{eq:expansion-1}
  T_0^\infty(1) &= T_0^{n-1}(1) + T_{n}^\infty(1), \\
  \label{eq:expansion-2}
  T_0^\infty(2) &= T_0^{n-1}(2) + T_{n}^\infty(1) \, T_0^{n-1}(1) + T_{n}^\infty(2) , \\
  \label{eq:expansion-3}
  T_0^\infty(3) &= T_0^{n-1}(3) + T_{n}^\infty(1) \, T_0^{n-1}(2) + T_{n}^\infty(2) \, T_0^{n-1}(1) + T_{n}^\infty(3) .
\end{align}
\end{subequations}
%The above equalities can be trivially interpreted in what follows to mean that if we have $r$ non-zero digits to consider in a range of positions from~$b^0$ to~$b^\infty$, then we can split the range in two non-overlapping ranges and consider $t$ and $(r-t)$ non-zero digits in each of them and consider all such splittings.

As we will apply \RefLem{lem:split} to a product function, $g(a_1,\ldots,a_r) = g(a_1) \cdots g(a_r)$ it is useful to obtain the following result.
\begin{lemma}
If the function $g(a_1,\ldots,a_r)$ is of product form $g_1(a_1) \cdots g_r(a_r)$, then $T_0^{n-1}(r)(g)$, with $n$ finite, can be calculated in $O(n r)$.
\end{lemma}

% \begin{proof}
%   The proof is a straightforward expansion of the sum operator in case of a product function.
%   We find
%   \begin{align*}
%     T_m^M(r)\!\left( \prod_{i=1}^r g_i(a_i) \right)
%     &=
%     \sum_{a_1=m+r-1}^M g_1(a_1) \sum_{a_2=m+r-2}^{a_1-1} g_2(a_2) \; \cdots \sum_{a_r=m}^{a_{r-1}-1} g_r(a_r)
%     .
%   \end{align*}
%   Now look at a two sum system:
%   \begin{multline*}
%     \sum_{a_1=1}^M g_1(a_1) \sum_{a2=0}^{a_1-1} g_2(a_2)
%     =
%     g_1(1) \Bigl( g_2(0) \Bigr)
%     +
%     g_1(2) \Bigl( g_2(0) + g_2(1) \Bigr)
%     \\+
%     g_1(3) \Bigl( g_2(0) + g_2(1) + g_2(2) \Bigr)
%     +
%     \cdots
%     +
%     g_1(M) \Bigl( g_2(0) + \cdots + g_2(M-1) \Bigr)
%     ,
%   \end{multline*}
%   then it is clear that by storing the value of the running intermediate sum this double sum only costs $O(M)$ instead of $O(M^2)$.
%   This can be done recursively over all $r$ sums to arrive at a complexity of $O((M-m+1) \, r)$ where $(M-m+1)$ is the number of values for $a_1$.
%   To store all running intermediate sums we need to store $r$ values.
% \end{proof}

The proof of the lemma follows from the number of operations needed in Algorithm~\ref{alg:calcsums}.

\begin{algorithm}
\caption{Compute $S_1 = T_0^{n-1}(r)(g_1(a_1) \cdots g_r(a_r))$ in $O(n r)$ operations}\label{alg:calcsums}
\begin{algorithmic}
\STATE Initialize $S_1=0$, \ldots, $S_r=0$
\FOR{$a_r=0$ \TO $n-r$}
  \STATE $S_r = S_r + g_r(a_r)$
  \FOR{$t=1$ \TO $r-1$}
    \STATE $S_{r-t} = S_{r-t} + S_{r-t+1} \, g_{r-t}(a_r + t)$
  \ENDFOR
\ENDFOR
\RETURN $S_1$
\end{algorithmic}
\end{algorithm}

At the end of this algorithm we have the post conditions:
\begin{align*}
  S_t
  &=
  T_0^{n-t}(r-t+1)\!\left( \prod_{i=t}^r g_i(a_i) \right)
  ,
  &
  \text{for } t = 1, \ldots, r.
\end{align*}
With a slight modification we can calculate all values of
\begin{align}
  \label{eq:allS}
  S_t
  &=
  T_0^{n-1}(r-t+1)\!\left( \prod_{i=t}^r g_i(a_i) \right)
  ,
  &
  \text{for } t = 1, \ldots, r.
\end{align}
For this we just let the outer loop run up to $n-1$ and make a modification in the inner loop to only conditionally update the value of $S_{r-t}$ as long as $a_r < n-t$.
The modified algorithm can be found in \RefAlg{alg:calcsums-all}.
This algorithm is still $O(n r)$.

\begin{algorithm}
\caption{Compute all $S_t = T_0^{n-1}(r-t+1)(g_t(a_t) \cdots g_r(a_r))$, for $t=1,\ldots,r$, in $O(n r)$}\label{alg:calcsums-all}
\begin{algorithmic}
\STATE Initialize $S_1=0$, \ldots, $S_r=0$
\FOR{$a_r=0$ \TO $n-1$}
  \STATE $S_r = S_r + g_r(a_r)$
  \FOR{$t=1$ \TO $\min(r,n-a_r)-1$}
    \STATE $S_{r-t} = S_{r-t} + S_{r-t+1} \, g_{r-t}(a_r + t)$
  \ENDFOR
\ENDFOR
\RETURN $(S_1, \ldots, S_r)$
\end{algorithmic}
\end{algorithm}

\subsection{A general algorithm for $x$ having a fixed base~$b$ precision of~$n$}\label{sec:omega-alpha-computation}

We now consider calculating $\omega_\alpha(x)$ in base $b$ for $x \in [0,1)$ which can be represented exactly with $n$~digit precision in base~$b$: $x = (0.\xi_1 \xi_2 \ldots \xi_n)_b = \sum_{i=1}^n \xi_i \, b^{-i}$.
That is, $x$ is actually a rational number $v/b^n$, $0 \le v < b^n$.
This is exactly the situation that occurs in the component-by-component construction of \RefSec{sec:cbc}, as the $v_n$ function\RefEq{eq:vn} exactly maps the Laurent series over $\F_b((\x^{-1}))$ to rationals $v/b^n$ with $0 \le v < b^n$.

\begin{theorem}\label{thm:ndigits}
  Let $\alpha, b \ge 2$ be integers. Then for any $x = v b^{-n}$ with $n \ge 1$ and $0 \le v < b^n$, the value of $\omega_\alpha(v b^{-n})$ can be computed in at most $O(\alpha n)$ operations as follows: calculate the vectors
  \begin{align*}
    \vec{T}(x)
    &\hphantom{:}=
    ( T_{\alpha-1}, \ldots, T_1 )
    \\
    &:= 
    \left(
      T^{n-1}_{0}(\alpha-1)\left( \prod_{i=1}^{\alpha-1} b^{-(a_i+1)} z(x,a_i) \right),
      \;\ldots,\,
      T^{n-1}_0(1)(b^{-(a_{\alpha-1}+1)} z(x,a_{\alpha-1}))
     \right)
    ,
    \\
    \tilde{\vec{T}}(x)
    &\hphantom{:}=
    ( \tilde{T}_\alpha, \ldots, \tilde{T}_1 )
    \\
    &:= 
    \left(
      T^{n-1}_0(\alpha)\left( b^{a_\alpha}[a_{\alpha} < \beta(x)-1] \prod_{i=1}^\alpha b^{-(a_i+1)} z(x,a_i) \right)
      , 
      \;\ldots,\,
    \right.
    \\ 
    & 
    \hspace{4.8cm}
    \left. \vphantom{\left( \prod_{i=1}^\alpha \right)}
      T^{n-1}_0(1)\left( b^{a_\alpha}[a_{\alpha} < \beta(x)-1] b^{-(a_\alpha+1)} z(x,a_\alpha) \right)
    \right)
    ,
  \end{align*}
  which can both be computed by \RefAlg{alg:calcsums-all}, and set
  \begin{align*}
    \vec{C}
    &=
    (C_0,\ldots,C_{\alpha-1})
    := 
    \left( b^{-nt} \prod_{i=1}^t \frac{b-1}{b^i-1} \right)_{t=0,\ldots,\alpha-1},
    \\
    \bar{\vec{C}}
    &=
    (\bar{C}_0, \bar{C}_1, \ldots,\bar{C}_{\alpha-1})
    := 
    ( C_0, C_0 + C_1, \ldots, C_0 + \cdots + C_{\alpha-1})
    ,
  \end{align*}
  where $C_0 = \bar{C}_0 = 1$,
  then for $0 \le v < b^n$
  \begin{align*}
    \omega_\alpha(vb^{-n})
    &=
    \begin{cases}
      \bar{\vec{C}}_{0:\alpha-2} \cdot \vec{T}(vb^{-n}) + (\bar{C}_{\alpha-1} - 1) + \vec{C} \cdot \tilde{\vec{T}}(vb^{-n})
        & \text{if } 0 < v < b^n, \\[1mm]
      \displaystyle \sum_{r=1}^{\alpha-1} \prod_{i=1}^r \frac{b-1}{b^i-1} +
        \frac{b-1}{b^\alpha-b} \prod_{i=1}^{\alpha-1} \frac{b-1}{b^i-1} & \text{if } v = 0 ,
    \end{cases}
  \end{align*}
  where $\vec{a} \cdot \vec{b}$ denotes the dot product and for $x = (0.\xi_1 \xi_2 \ldots \xi_n)_b$ we set
  \begin{align*}
    z(x, a_i)
    &=
    \begin{cases}
      b-1 & \text{if } \xi_{a_i+1} = 0, \\
      -1  & \text{if } \xi_{a_i+1} \ne 0 ,
    \end{cases}
    &\text{and}&&
    \beta(x)
    &=
    %\begin{cases}
      -\floor{\log_b(x)}%, & \text{if } 0 < x < 1 
      %,\\
    %  \infty , & \text{if } x = 0 
    .
    %\end{cases}
  \end{align*}
\end{theorem}
\begin{proof}
We start from expression\RefEq{eq:omega-step1}. For ease of manipulation we consider two different cases of the base~$b$ expansions for integer~$k > 0$:
  \begin{enumerate}
    \item Integers $k$ which have between $1$ and $(\alpha-1)$ non-zero digits in base~$b$:
      \begin{align*}
        k &= \sum_{i=1}^{\nzd{k}} \kappa_{a_i} \, b^{a_i}
        ,
        \qquad \text{where } 1 \le \nzd{k} \le \alpha-1.
      \end{align*}
    \item Integers $k$ which have $\alpha$ or more non-zero digits in base~$b$:
      \begin{align*}
        k &= \sum_{i=1}^{\nzd{k}} \kappa_{a_i} \, b^{a_i}
           = \sum_{i=1}^\alpha \kappa_{a_i} \, b^{a_i} + k'
        ,
        \qquad \text{where } \nzd{k} \ge \alpha \text{ and } 0 \le k' < b^{a_\alpha} .
      \end{align*}
  \end{enumerate}
  As such we consider, for $0 < x < 1$,
  \begin{multline*}
    \omega_\alpha(x)
    =
    \sum_{\substack{k=1 \\[0.5mm] \nzd{k} < \alpha \\[0.5mm] k = \sum_{i=1}^{\nzd{k}} \kappa_{a_i} \, b^{a_i}}}^\infty
      \prod_{i=1}^{\nzd{k}} b^{-(a_i + 1)} \exp(\twopii \kappa_{a_i} \xi_{a_i+1} / b)
    \\[-2mm] +
    \sum_{\substack{k=1 \\[0.5mm] \nzd{k} \ge \alpha \\[0.5mm] k = \sum_{i=1}^{\alpha} \kappa_{a_i} \, b^{a_i} + k' \\[0.5mm] 0 \le k' < b^{a_\alpha}}}^\infty
      \wal_{k'}(x)
      \prod_{i=1}^\alpha b^{-(a_i + 1)} \exp(\twopii \kappa_{a_i} \xi_{a_i+1} / b)
    .
  \end{multline*}
  We will now expand these outer sums as in\RefEq{eq:expand}, but first define
  \begin{align}\label{eqn:zhelp}
    z(x, a_i)
    &:=
    \sum_{\kappa_{a_i}=1}^{b-1} \exp(\twopii \kappa_{a_i} \xi_{a_i+1} / b)
    =
    \begin{cases}
      b-1 & \text{if } \xi_{a_i+1} = 0, \\
      -1  & \text{if } \xi_{a_i+1} \ne 0 ,
    \end{cases}
  \end{align}
  to move all the independent $\kappa_{a_i}$ sums, cf.\RefEq{eq:expand}, into the product function.
  Further, denote by $\beta(x)$ the power of $b^{-1}$ of the first non-zero digit in the base~$b$ expansion of $x \in [0,1)$, then the sum over $k'$ for case 2 becomes
  \begin{align*}
    \sum_{k'=0}^{b^{a_\alpha}-1} \wal_{k'}(x)
    &=
    \begin{cases}
      b^{a_\alpha} & \text{if } a_\alpha < \beta(x)-1 \text{, i.e., } x = (0.\underbrace{0 \ldots\ldots 0}_{\text{at least $a_\alpha$}}***\ldots)_b,  \\
      0        & \text{otherwise}
    \end{cases}
    \\&=:
    b^{a_\alpha} [a_\alpha < \beta(x)-1],
  \end{align*}
  where the last line uses Iverson notation.
  Introducing the sum operator\RefEq{eq:sum-operator} we obtain
  \begin{align*}
    \omega_\alpha(x)
    &=
    \sum_{r=1}^{\alpha-1} T_0^\infty(r)\!\!\left(
      \prod_{i=1}^{r} b^{-(a_i+1)} z(x,a_i)
    \right)
    +
    T_0^\infty(\alpha)\!\!\left(
      b^{a_\alpha} [a_\alpha < \beta(x)-1]
      \prod_{i=1}^{\alpha} b^{-(a_i+1)} z(x,a_i)
    \right)
    .
  \end{align*}
  Since our function is a product function, it is convenient to only deal with the operators, which then shortens the notation.

  We now deal with the two cases separately.
  For case~1, $1 \le r \le \alpha - 1$, we apply \RefLem{lem:split} and manipulate the following expression
  \begin{align*}
    \sum_{r=1}^{\alpha - 1} T_0^\infty(r)
    &=
    \sum_{r=1}^{\alpha - 1}
    \sum_{t=0}^r
      T_n^\infty(r-t) \,
      T_0^{n-1}(t)
    =
    \sum_{t=1}^{\alpha - 1}
      \left( \sum_{r=t}^{\alpha - 1} T_n^\infty(r-t) \right) T_0^{n-1}(t)
    + \sum_{r=1}^{\alpha-1} T_n^\infty(r)
    .
   \end{align*}
   By assumption of the $n$~digit base~$b$ precision of~$x$ the $T_n^\infty$ sums do not depend on $x$.
   As we show next, they can be calculated off line in closed form.
   That means we are left to deal with the $T_0^{n-1}(t)(g_{r-t+1}(a_{r-t+1}) \cdots g_r(a_r))$ for $t=1,\ldots,\alpha-1$.
   We can use \RefAlg{alg:calcsums} for each of these terms, but as they are nested, we can use \RefAlg{alg:calcsums-all} to calculate them all at once in time $O(\alpha n)$ upon calculating $T_0^{n-1}(\alpha-1)$.
   The $T_n^\infty$ sums are given by:
%
%   As we assumed that at most the leading $n$ digits in the base~$b$ expansion of $x \in [0,1)$ are non-zero, the $T_n^\infty$ sums do not depend on $x$, in fact:
  \begin{align}
    \nonumber
    T_n^\infty(t)\!\left( \prod_{i=1}^t b^{-(a_i+1)} (b-1) \right)
    &=
    (b-1)^t b^{-t}
    \sum_{a_1=n+t-1}^\infty b^{-a_1}
    \sum_{a_2=n+t-2}^{a_1-1} b^{-a_2}
    \cdots
    \sum_{a_t=n}^{a_{t-1}-1} b^{-a_t}
    \\
    \nonumber
    &=
    (b-1)^t b^{-t}
    \sum_{a_t=n}^\infty b^{-a_t}
    \cdots
    \sum_{a_2=a_3+1}^{\infty} b^{-a_2}
    \sum_{a_1=a_2+1}^{\infty} b^{-a_1}
%    \\
%    \nonumber
%    &=
%    (b-1)^t b^{-t} b^{-(n-1)t} \prod_{i=1}^t (b^i-1)^{-1}
    \\
    \label{eq:constants}
    &=
    b^{-nt} \prod_{i=1}^t \frac{b-1}{b^i-1}
    .
  \end{align}

   For case~2, $\nzd{k} \ge \alpha$, we can also apply \RefLem{lem:split} to obtain
   \begin{align*}
     T_0^\infty(\alpha)
     &=
     \sum_{t=0}^\alpha T_n^\infty(t) T_0^{n-1}(\alpha-t)
     \\&=
     T_0^{n-1}(\alpha) + T_n^\infty(1) T_0^{n-1}(\alpha-1) + \cdots +
     T_n^\infty(\alpha-1) T_0^{n-1}(1) + T_n^\infty(\alpha)
     ,
   \end{align*}
   which is applied to the function
   \begin{align*}
     b^{a_\alpha} [a_\alpha < \beta(x)-1]
     \left( \prod_{i=1}^{\alpha} b^{-(a_i+1)} z(x,a_i) \right)
     .
   \end{align*}
   The $T_n^\infty$ sums here become
   \begin{gather*}
     T_n^\infty(t)\left(\prod_{i=1}^t b^{-(a_i+1)} (b-1) \right)
     =
     b^{-nt} \prod_{i=1}^t \frac{b-1}{b^i-1}
     ,
     \qquad
     \text{for $t < \alpha$,}
     \\
     \text{and}\qquad
     T_n^\infty(\alpha)\left(b^{a_\alpha} [a_\alpha < \beta(x)-1] \prod_{i=1}^\alpha b^{-(a_i+1)} (b-1)\right)
     =
     0
     .
   \end{gather*}
For $x\ne0$ the condition $[a_\alpha < \beta(x)-1]$ makes it such that $T_n^\infty(\alpha) = 0$ as $a_\alpha \ge n$ (and $\beta(x) \le n$ by assumption).
%   Without the condition this sum would be unbounded.
The other $T_n^\infty$ values are the same as for case~1, and we can use the closed form\RefEq{eq:constants}. Also here we use \RefAlg{alg:calcsums-all} to calculate all the sums $T_0^{n-1}(\alpha-t)$ in $O(\alpha n)$ upon calculating $T_0^{n-1}(\alpha)$.

When $x=0$ there is no need to consider splitting at a given $n$. To obtain $T_0^\infty(\alpha)$ we can use a similar derivation as for\RefEq{eq:constants} to  obtain a closed form:
   \begin{align*}
     T_0^\infty(\alpha)\left(b^{a_\alpha} \prod_{i=1}^\alpha b^{-(a_i+1)} (b-1)\right)
     &=
     \sum_{a_\alpha=0}^\infty b^{-1} (b-1) \; T_{a_\alpha+1}^\infty(\alpha-1)\left(\prod_{i=1}^{\alpha-1} b^{-(a_i+1)} (b-1)\right)
     \\
     &=
     \sum_{a_\alpha=0}^\infty b^{-1} (b-1) \; b^{-(a_\alpha+1)(\alpha-1)} \prod_{i=1}^{\alpha-1} \frac{b-1}{b^i-1}
     \\
     &=
     \frac{b-1}{b^\alpha-b} \prod_{i=1}^{\alpha-1} \frac{b-1}{b^i-1}
     .
   \end{align*}
Again\RefEq{eq:constants} can be used to calculate the $T_0^\infty(r)$ for $r=1,\ldots,\alpha-1$. This completes the proof.
\end{proof}

\subsection{Explicit forms for arbitrary $x$ and small $\alpha$}

\RefThm{thm:ndigits} uses the fact that at most the first $n$ digits of the coordinates of the polynomial lattice rule can be non-zero; it is hence not surprising that the resulting computational complexity depends on $n$.
Here we take a similar approach, but explicitly look at the non-zero digits of~$x$; this will turn out to be a favorable approach in case of $b=2$, for which we find explicit expressions in \RefCol{col:base2}.
We will use the following similar notation as was set up in the beginning of \RefSec{sec:funspace}: Let the non-zero digits base~$b$ expansion of $x = (0.\xi_1 \xi_2 \ldots)_b \in [0,1)$ be given by
\begin{align*}
  x
  &=
  \sum^{\nzd{x}}_{i=1} \xi_{a_i} b^{-a_i}
  ,
\end{align*}
where $1 \leq a_1 < \dots < a_{\nzd{x}}$, $\xi_{a_i} \in \left\{ 1, \ldots, b-1 \right\}$. In particular, we will see that the power of $b^{-1}$ for the most significant digit of $x$, i.e., $a_1$, plays a pivotal role.
For $x=0$ we set $a_1 = \infty$ and $\#x=0$.

\begin{theorem}\label{thm:arbitrary-x}
  For $x \in [0,1)$ with non-zero digit base~$b$ expansion
  \begin{align*}
    x
    &=
    \sum^{\nzd{x}}_{i=1} \xi_{a_i} b^{-a_i}
    ,
    &
    1 \leq a_1 < \dots < a_{\nzd{x}}, \quad \xi_{a_i} \in \{1,\ldots,b-1\},
  \end{align*}
  we have
  \begin{align*}
    \omega_2(x)
    &=
    s_1(x) + \tilde{s}_2(x)
    ,
    \\
    \omega_3(x)
    &=
    s_1(x) + s_2(x) + \tilde{s}_3(x)
    ,
  \end{align*}
  where
  \begin{align*}
    s_1(x)
    &:=
    1 - b \sum_{j=1}^{\nzd{x}} b^{-a_j}
    ,
    \\
    s_2(x)
    &:=
    \frac{1}{b+1}
    -
    b (b-2)
    \frac{1}{2} \left(
      \left( \sum_{j=1}^{\nzd{x}} b^{-a_j} \right)
      \left( \sum_{j=1}^{\nzd{x}} b^{-a_{j}} \right)
      - \sum_{j=1}^{\nzd{x}} b^{-2a_j}
    \right)
    \\
    &\qquad
    -
    b (b-1)
    \left( \frac{1}{b-1} - \sum_{j=1}^{\nzd{x}} b^{-a_j} \right)
    \left( \sum_{j=1}^{\nzd{x}} b^{-a_j} \right),
\end{align*}
and for $x \ne 0$ we have
\begin{align*}
    \tilde{s}_2(x)
    &:=
    b^{-1} - 2 b^{-a_1} - b^{-(a_1+1)} - ( a_1 b - a_1 - b ) \sum_{j=1}^{\nzd{x}} b^{-a_j}
    ,
    \\
    \tilde{s}_3(x)
    &:=
  \Bigl(b^{-a_1+1}(a_1b - a_1 - b + 2) -1 \Bigr)
    \sum_{j=2}^{\nzd{x}} b^{-a_j}
  - b^{-1} ( a_1 b - a_1 - b + 1 ) b^{-2 a_1} s_1(b^{a_1} x \bmod{1})
  \\
  &\qquad
    + b^{-1} ( a_1 b - a_1 - b ) b^{-2a_1} s_2(b^{a_1} x \bmod{1})
    ,
  \end{align*}
  where $b^{a_1} x \bmod{1} = b^{a_1} x - \xi_{a_1}$.
  For $x=0$ we set $\tilde{s}_2(0) = b^{-1}$ and $\tilde{s}_3(0) = b^{-1}(b+1)^{-2}$.
\end{theorem}
\begin{proof}
  We start in exactly the same way as in \RefThm{thm:ndigits}, that is, we split $\omega_\alpha(x)$ into $\alpha$ parts (cf.\ the $\alpha-1$ parts in case~1 plus the case~2 case):
  \begin{align*}
    \omega_\alpha(x)
    =
    \sum_{k=1}^\infty r_\alpha(k) \, \wal_k(x)
    &=
    \sum_{r=1}^{\alpha-1} s_r(x) + \tilde{s}_\alpha(x)
    ,
  \end{align*}
  where $s_r(x)$ contains all $k$ with exactly $r$ digits non-zero and $\tilde{s}_\alpha(x)$ contains all $k$ with at least $\alpha$ digits non-zero.
We only show the derivation of the formulae for $s_1$ and $s_2$ as examples.
The ones for $\tilde s_2$ and $\tilde s_3$ can be obtained similarly.
%  Instead of using $a_i$ indices to express $k$ we will be using $\ell'$ and $\ell$ as we here use the $a_i$ indices for the non-zero digits of $x$.
With $z$ as in \eqref{eqn:zhelp} we find 
\begin{align*}
  s_1(x)
  &=
  \sum_{\ell=0}^\infty
  b^{-(\ell+1)} z(x, \ell)
  \\
  &=
  (b-1)
  \sum_{\ell=0}^\infty
  b^{-(\ell+1)}
  -
  ((b-1) + (-1))
  \sum_{\ell=0}^\infty
  b^{-(\ell+1)} [\xi_{\ell+1}\ne 0]
  \\
  &=
  1
  -
  b
  \sum_{j=1}^{\nzd{x}} b^{-a_j}
  .
\end{align*}
Likewise for $s_2$:
\begin{align}
  \nonumber
  s_2(x)
  &=
  \sum_{\ell'=0}^\infty
  b^{-(\ell'+1)} z(x, \ell')
  \sum_{\ell=\ell'+1}^\infty
  b^{-(\ell+1)} z(x, \ell)
  \\ \nonumber
  &=
  \frac{1}{b+1}
  \\
  &\quad \label{eq:line2} \tag{*}
  -
  b (b-2) \sum_{\ell'=0}^\infty b^{-(\ell'+1)} [\xi_{\ell'+1}\ne 0]
          \sum_{\ell=\ell'+1}^\infty b^{-(\ell+1)} [\xi_{\ell+1}\ne 0]
  \\
  &\quad \label{eq:line3} \tag{**}
  -
  b (b-1) \sum_{\ell'=0}^\infty b^{-(\ell'+1)} [\xi_{\ell'+1}\ne 0]
          \sum_{\ell=\ell'+1}^\infty b^{-(\ell+1)} [\xi_{\ell+1}= 0]
  \\
  &\quad \label{eq:line4} \tag{**}
  -
  b (b-1) \sum_{\ell'=0}^\infty b^{-(\ell'+1)} [\xi_{\ell'+1}= 0]
          \sum_{\ell=\ell'+1}^\infty b^{-(\ell+1)} [\xi_{\ell+1}\ne 0]
  .
\end{align}
  This is a combinatorial formulation in terms of the possibilities for the digits of $x$.
  %Note that for $b=2$ the first sum vanishes.
  The two last lines, marked by\RefEq{eq:line3}, can be combined and interpreted as summing over all possible pairs of digits of $x$ of which exactly one is non-zero.
  This then simplifies to two decoupled sums since a digit cannot be at the same time zero and non-zero:
\begin{align*}
  \sum_{\ell'=0}^\infty b^{-(\ell'+1)} [ \xi_{\ell'+1}=0 ]
  \sum_{\ell=0}^\infty b^{-(\ell+1)} [ \xi_{\ell+1} \ne 0 ]
  &=
  \left( \sum_{\ell'=0}^\infty b^{-(\ell'+1)} - \sum_{j=1}^{\nzd{x}} b^{-a_j} \right)
  \left( \sum_{j=1}^{\nzd{x}} b^{-a_j} \right)
  \\
  &=
  \left( \frac{1}{b-1} - \sum_{j=1}^{\nzd{x}} b^{-a_j} \right)
  \left( \sum_{j=1}^{\nzd{x}} b^{-a_j} \right)
  .
\end{align*}
%which simplifies to $(1 - x) x$ for $b=2$. In full, when $b=2$ we find $s_2(x) = 3^{-1} - 2 (1 - x) x$.
%
%When $b \ne 2$ we also need the second double sum.
The other double sum, marked by\RefEq{eq:line2}, can also be interpreted combinatorially: the sum is taken over all ordered pairs of non-zero digits of $x$.
We can write:
\begin{align*}
  \sum_{\ell'=0}^\infty b^{-(\ell'+1)} [\xi_{\ell'+1}\ne 0]
          \sum_{\ell=\ell'+1}^\infty b^{-(\ell+1)} [\xi_{\ell+1}\ne 0]
  &=
  \sum_{j=1}^{\nzd{x}} b^{-a_j} \sum_{j'=j+1}^{\nzd{x}} b^{-a_{j'}}
  \\
  &=
  \frac{1}{2} \left(
    \left( \sum_{j=1}^{\nzd{x}} b^{-a_j} \right)
    \left( \sum_{j=1}^{\nzd{x}} b^{-a_{j}} \right)
    - \sum_{j=1}^{\nzd{x}} b^{-2a_j}
  \right)
  .
\end{align*}
Thus
\begin{align*}
  s_2(x)
  &=
  \frac{1}{b+1}
  \\
  &\quad
  -
  b (b-2)
  \frac{1}{2} \left(
    \left( \sum_{j=1}^{\nzd{x}} b^{-a_j} \right)
    \left( \sum_{j=1}^{\nzd{x}} b^{-a_{j}} \right)
    - \sum_{j=1}^{\nzd{x}} b^{-2a_j}
  \right)
  \\
  &\quad
  -
  b (b-1)
  \left( \frac{1}{b-1} - \sum_{j=1}^{\nzd{x}} b^{-a_j} \right)
  \left( \sum_{j=1}^{\nzd{x}} b^{-a_j} \right)
  .
\end{align*}
\end{proof}

The case where $b$ equals $2$ is of greatest practical importance, since in that case the matrix-vector product\RefEq{eq:mv} over $\F_b$ to generate the nodes of the QMC rule can be calculated most efficiently by using the bitwise operations of the computer. Additionally
\begin{align*}
  \sum_{j=1}^{\nzd{x}} b^{-a_j} = x
  \qquad\text{when}\qquad b=2.
\end{align*}
By specializing the previous result to $b=2$ we obtain the following explicit formulae.
\begin{corollary}\label{col:base2}
  For base $b=2$ we obtain the following explicit results:
  \begin{align*}
    \omega_2(x)
    &=
    s_1(x) + \tilde{s}_2(x)
    ,
    \\
    \omega_3(x)
    &=
    s_1(x) + s_2(x) + \tilde{s}_3(x)
    ,
  \end{align*}
  where
  \begin{align*}
    s_1(x) &= 1 - 2 x
    ,
    &
    s_2(x) &= 1/3 - 2 (1-x) x
    ,
    \\
    \tilde{s}_2(x) &= (1 - 5 t_1) / 2 + (2 - a_1) x
    ,
    &
    \tilde{s}_3(x) &= (1 - 43 t_2)/18 + (5 t_1 - 1) x - (2 - a_1) x^2
    ,
  \end{align*}
  with, for $0 < x < 1$,
  \begin{align*}
    a_1 &= - \floor{\log_2(x)}
    ,
    &
    t_1 &:= 2^{-a_1}
    ,
    &
    t_2 &:= 2^{-2 a_1}
    ,
  \end{align*}
  and $a_1 = 0$, $t_1 = 0$ and $t_2 = 0$ when $x=0$.
\end{corollary}
\begin{proof}
  To obtain $\tilde{s}_3(x)$ we note that, for $x \ne 0$,
  \begin{align*}
    b^{a_1} x \bmod{1} 
    &=
    b^{a_1} x - \xi_{a_1}
    =
    x / t_1 - 1,
  \end{align*}
  and thus
  \begin{align*}
    s_1(x / t_1 - 1)
    &=
    3 - 2 x / t_1
    ,
    &
    s_2(x / t_1 - 1)
    &=
    13/3 - 6 x / t_1 + 2 x^2 / t_1^2
    .
  \end{align*}
\end{proof}

%\begin{figure}
%  \centering
%  \includegraphics{omegaplots}
%  \caption{The corresponding Walsh ``kernels'' for $\alpha=1$, $2$ and $3$ in base $2$}
%\end{figure}

\section{Numerical tests}

We compare the explicit construction from \cite{Dic2008}, with the CBC algorithm based on (fast) circular convolution presented in this paper, i.e., \RefAlg{alg:fast-cbc}. From \cite{Dic2008} we note that, to obtain higher order digital nets of high quality, the underlying point sets in the construction should have small values of $t$. Consequently, we use Niederreiter-Xing points generated by Pirsic's implementation, see \cite{Pir2002}, to obtain the digital $(t', m, sd)$-nets.
In Table~\ref{tbl:alpha-2} we present a typical result for $b=2$, $\alpha=2$ and $s=5$ and two choices of weights $\gamma_j=0.9^j$ and $\gamma_j=j^{-2}$.
The numerical data in all our tests shows that the new construction produces better results.

\begin{table}
  \centering
  \begin{tabular}{c|cc}
    \hline
    $\gamma_j=0.9^j$ & $e_\text{CBC}$ & $e_\text{explicit}$ \\
    \hline
    $m=5$            & $0.9291$             & $1.0930$        \\
    $m=6$            & $0.4085$             & $0.4259$        \\
    $m=7$            & $0.1778$             & $0.1984$        \\
    $m=8$            & $0.0747$             & $0.0980$        \\
    $m=9$            & $0.0312$             & $0.0403$        \\
    $m=10$           & $0.0128$             & $0.0168$        \\
    $m=11$           & $0.0052$             & $0.0071$        \\
    $m=12$           & $0.0020$             & $0.0027$        \\
    \hline
  \end{tabular}
  \qquad
  \begin{tabular}{c|cc}
    \hline
    $\gamma_j=j^{-2}$ & $e_\text{CBC}$ & $e_\text{explicit}$ \\
    \hline
    $m=5$            & $0.028917$             & $0.096254$        \\
    $m=6$            & $0.009912$             & $0.014542$        \\
    $m=7$            & $0.003427$             & $0.005895$        \\
    $m=8$            & $0.001175$             & $0.002356$        \\
    $m=9$            & $0.000406$             & $0.000827$        \\
    $m=10$           & $0.000139$             & $0.000290$        \\
    $m=11$           & $0.000046$             & $0.000091$        \\
    $m=12$           & $0.000014$             & $0.000034$        \\
    \hline
  \end{tabular}
\caption{Comparison of the worst-case errors of CBC construction and explicit construction}\label{tbl:alpha-2}
\end{table}

For reference we conclude the paper with tables showing the generating vectors and worst case errors of higher order polynomial lattice rules in base~$2$ constructed using the new algorithm.
All polynomials are given by their canonical integer representation which is the polynomial evaluated at $\x=b=2$.
The results can be found in Table~\ref{tbl:vectors-alpha2} and Table~\ref{tbl:vectors-alpha3} for $\alpha=2$ and $\alpha=3$ respectively. 

\begin{table}
  \centering \small
  \begin{tabular}{c|cccccccccc}
    \hline
    \multicolumn{6}{l}{$b=2$, $m=10$, $\alpha=2$: $n=20$, $p=1179649$}
    \\
    \hline
    $j$ &
    1 &
    2 &
    3 &
    4 &
    5 
\\
    $q_j$ &
    453270 &
    920860 &
    324514 &
    394664 &
    106142
\\
    $e$ &
    2.14e-6 &
    4.55e-5 &
    6.27e-4 &
    3.75e-3 &
    1.30e-2
 \\
    \hline 
    $j$ &
    6 &
    7 &
    8 &
    9 &
    10 
\\
    $q_j$ &     
    587632 &
    279628 &
    676057 &
    626366 &
    856775  \\ $e$ &     3.39e-2 &
    7.45e-2 &
    1.43e-1 &
    2.51e-1 &
    4.08e-1 \\ \hline
  \end{tabular}
  \\[2mm]
  \begin{tabular}{c|cccccccccc}
    \hline
    \multicolumn{6}{l}{$b=2$, $m=12$, $\alpha=2$: $n=24$, $p=28311553$}
    \\
    \hline
    $j$ &
    1 &
    2 &
    3 &
    4 &
    5 
\\
    $q_j$ &
    2028384 &
    13051202 &
    839202 &
    14647583 &
    6874738
 \\
    $e$ &
    1.34e-7 &
    3.44e-6 &
    6.58e-5 &
    4.72e-4 &
    2.02e-3
 \\
    \hline 
    $j$ &
    6 &
    7 &
    8 &
    9 &
    10 
\\
    $q_j$ &     6522492 &
    13569662 &
    9821234 &
    10570369 &
    406897 \\  $e$ &     6.09e-3 &
    1.45e-2 &
    2.97e-2 &
    5.46e-2 &
    9.19e-2 \\ \hline
  \end{tabular}
\caption{Higher order rules up to $10$ dimensions for $b=2$, $\gamma_j = 0.9^j$ and $\alpha=2$}\label{tbl:vectors-alpha2}
\end{table}

\begin{table}
  \centering \small
  \begin{tabular}{c|cccccccccc}
    \hline
    \multicolumn{6}{l}{$b=2$, $m=7$, $\alpha=3$: $n=21$, $p=2621441$}
    \\
    \hline
    $j$ &

    1 &
    2 &
    3 &
    4 &
    5 
  \\
    $q_j$ &
    1492861 &
    1022044 &
    1785216 &
    215936 &
    1978368
 \\
    $e$ &
    2.02e-6 &
    5.24e-4 &
    8.20e-3 &
    4.05e-2 &
    1.22e-1
  \\
    \hline 
    $j$ &
    6 &
    7 &
    8 &
    9 &
    10 
  \\
    $q_j$ &     1197580 &
    1837814 &
    485609 &
    1636853 &
    48810 \\ $e$ &     2.82e-1 &
    5.54e-1 &
    9.80e-1 &
    1.60    &
    2.48   \\ \hline
  \end{tabular}
  \\[2mm]
  \begin{tabular}{c|ccccc}
    \hline
    \multicolumn{6}{l}{$b=2$, $m=8$, $\alpha=3$: $n=24$, $p=28311553$}
    \\
    \hline
    $j$ &
    1 &
    2 &
    3 &
    4 &
    5 
  \\
    $q_j$ &
    10844342 &
    2604270 &
    5720893 &
    8141702 &
    3831799
     \\
    $e$ &
    2.51e-7 &
    8.85e-5 &
    2.43e-3 &
    1.45e-2 &
    4.95e-2
     \\   \hline
    $j$ &
    6 &
    7 &
    8 &
    9 &
    10 
  \\
$q_j$ & 3616803 & 15701694 &
    7750425 &
    2240926 & 493873  \\ $e$ &   1.21e-1 & 2.49e-1 &
    4.54e-1 &
    7.59e-1 & 1.19  \\ \hline
  \end{tabular}
\caption{Higher order rules up to $10$ dimensions for $b=2$, $\gamma_j = 0.9^j$ and $\alpha=3$}\label{tbl:vectors-alpha3}
\end{table}

% \section{Conclusion}
%
% This work fits in a line of research concerned with higher order
% polynomial lattice rules, see \cite{BDGP2011,DKPS2007,DP2007}. In \cite{BDGP2011}, it was shown how to modify the component-by-component algorithm to produce optimal
% convergence rates for a fixed, but arbitrarily large, range of
% smoothness parameters; the resulting algorithm, due to its structure,
% was named the \emph{component-by-component sieve} algorithm. It is trivial to see
% how the fast component-by-component algorithm presented in this paper
% can be adapted to reduce the computational cost of the
% component-by-component sieve algorithm. Consequently, one can obtain
% algorithms which are extensible in both, smoothness parameter and
% dimension, and, which is the contribution of the current paper, fast
% (or structured) versions of these algorithms.

%\bibliographystyle{abbrv} 
%\bibliography{mrabbrev,bib-abbrev}

\end{document}